\documentclass[10pt]{amsart}
\usepackage{amsmath}
\usepackage{amsthm}
\usepackage{color}
\usepackage{amsfonts,comment,booktabs,adjustbox}
\usepackage{amssymb}
\usepackage{cancel}
\usepackage{amscd}
\numberwithin{equation}{section}
\newtheorem{theorem}{Theorem}[section]
\newtheorem{cor}[theorem]{Corollary}
\newtheorem{conj}[theorem]{Conjecture}
\newtheorem{remark}[theorem]{Remark}
\newtheorem{proposition}[theorem]{Proposition}

\newtheorem{lemma}[theorem]{Lemma}
\theoremstyle{definition}

\newcommand\g{\mathfrak g}

\newcommand{\Z}{\mathbb Z}

\newcommand\C{\mathbb C}

\newcommand{\bea}{\begin{eqnarray}}
\newcommand{\eea}{\end{eqnarray}}
\begin{document}


  \title[]{ Vertex Algebras Related  to Regular Representations of $SL_2$  }

\author[Dra\v zen  Adamovi\'c]{Dra\v zen  Adamovi\' c}
 \address[Dra\v zen  Adamovi\' c]{Department of Mathematics \\ Faculty of Science  \\ University of Zagreb \\ 10 000 Zagreb, Croatia}
 \email[Dra\v zen  Adamovi\' c]{adamovic@math.hr}

\author[Antun Milas]{Antun Milas}
  \address[Antun Milas]{Department of Mathematics and Statistics  \\ SUNY-Albany \\  Albany NY 12222, USA}
 \email[Antun Milas]{amilas@albany.edu}


  \begin{abstract}  
 We construct a family of potentially quasi-lisse (non-rational) vertex algebras, denoted by $\mathcal{C}_p$, $p \geq 2$, which are closely related to the vertex algebra of chiral differential operators on $SL(2)$ at level $-2+\frac{1}{p}$. 
We prove that for $p = 3$, there is an isomorphism between $\mathcal{C}_3$ and the affine vertex algebra $L_{-5/3}(\mathfrak{g}_2)$ from Deligne's series. Moreover, we also establish isomorphisms between $\mathcal{C}_4$ and $\mathcal{C}_5$ and certain affine 
${W}$-algebras of types $F_4$ and $E_8$, respectively. In this way, we resolve the problem of decomposing certain conformal embeddings of affine vertex algebras into affine ${W}$-algebras. An important feature is that $\mathcal{C}_p$ is $\frac{1}{2} \mathbb{Z}_{\geq 0}$-graded with finite-dimensional graded subspaces and convergent characters. Therefore, for all $p \geq 2$, we show that the characters of $\mathcal{C}_p$ exhibit modularity, supporting the conjectural quasi-lisse property. 

 \end{abstract}
 \maketitle
\section{Introduction}


\subsection{Regular representations of vertex algebras}
The study of the regular representation of vertex algebras began with the work of Feigin and Parkhomenko \cite{FP} for $\frak{sl}_2$ and $\frak{sl}_3$, and more comprehensively by Frenkel and Styrkas \cite{FS} on their work on modified regular representation. An alternative approach involves the geometric construction of algebras of global chiral differential operators
$\mathcal{D}^{ch}_{X}$ on the cotangent bundle $T^*G$, which can be traced back to earlier works \cite{AG,GB,GB1}; see also \cite{Zhu} for connection between two approaches. A more recent perspective on chiral differential operators $\mathcal{D}^{ch}_{G,k}$ was undertaken by Arakawa, who explored them within the context of quasi-lisse vertex algebras and genus zero $S$-class theories in 4d $N=2$ SCFT  \cite{Arakawa}. 

These works are rooted in classical representation theory, starting with the regular representation $\mathcal{O}(G)$, where $G$ is connected, simply-connected and semisimple, that decomposes by the Peter–Weyl theorem
\begin{equation} \label{regular}
\mathcal{O}(G)=\bigoplus_{\lambda \in P_+} V_\lambda \otimes V_{\lambda^*},
\end{equation}
where $V_\lambda$ is finite-dimensional irreducible $\frak{g}={\rm Lie}(G)$-module of highest weight $\lambda$ and $\lambda^*=-w_0(\lambda)$ \footnote{There is also version of this result for 
$\mathcal{O}(Bw_0B)$ where $Bw_0B$ is the big Bruhat cell \cite{FS}.}.
In the affine Lie algebra setting, the algebra of function is replaced with the vertex algebra of chiral differential operators that decomposes as $\hat{\frak{g}} \times \hat{\frak{g}}$-module 
\begin{equation} \label{CDO}
\mathcal{D}^{ch}_{G,k} \cong \bigoplus_{\lambda \in P^+} L_{-k-h^\vee,\lambda}(\frak{g}) \otimes L_{k-h^\vee,\lambda^*}(\frak{g}),
\end{equation}
where $k \notin \mathbb{Q}$ is a generic level and $h^\vee$ is the dual Coxeter number. 

\subsection{Minimal and Weil representations} 
The oscillator, or Weil representation, of $\mathfrak{sp}(2n)$, also known as the metaplectic representation, is a fundamental tool in harmonic analysis and representation theory. Weil-type representations also exist for exceptional Lie algebras. For example, S. Gelfand \cite{Gelfand} studied a representation of the exceptional Lie algebra $\mathfrak{g}_2$, realized as differential operators acting on a suitable space of Laurent polynomials in three variables; see also \cite{Savin}.
These representations are minimal because their annihilator $J$ in the universal enveloping algebra $\mathcal{U}(\mathfrak{g})$ coincides with the Joseph ideal $\mathcal{J}_0$ - the unique completely prime ideal satisfying $Ass(\mathcal{J}_0) = \overline{\mathcal{O}}_{\text{min}}$.

Minimal and Weil-type representations also appear in the context of vertex algebras. It is known \cite{AM0} that for every affine vertex operator algebras $L_k(\frak{g})$, with $k = -\frac{h^\vee}{6} - 1$,  in the Deligne series 
\[
A_1 \subset A_2 \subset G_2 \subset D_4 \subset E_6 \subset E_7 \subset E_8,
\]
the Zhu algebra $A(L_k(\frak{g}))$  contains $\mathcal{U}(\mathfrak{g})/\mathcal{J}_0$ as one of its components, indicating that the minimal representations of $\mathfrak{g}$ also appear as the top component of certain $L_k(\frak{g})$-modules. 
Interestingly, this property also arises for vertex algebras outside Deligne's series. For the symplectic Lie algebra $\mathfrak{sp}_{2n}$, this occurs for the vertex algebra
$L_{-1/2}(\mathfrak{sp}_{2n})$ \cite{AM}, which is also the $\mathbb{Z}$-graded part of the  $\beta\gamma$-system $M_{(n)}$ of rank $n$.
All above vertex algebras are {\em quasi-lisse}, meaning their associated variety $X_{L_k(\frak{g})}:={\rm Specm}(R_{L_k(\frak{g})})$ admits finitely many symplectic leaves (as a Poisson 
variety).

\subsection{Our paper and main results}

In this paper, we construct a family of vertex algebras that, on one hand, deform the vertex algebra of regular representation of ${\mathfrak{sl}}_2$, and, on the other, exhibit characteristics of the Deligne series of vertex algebras at admissible level discussed above. Consequently, these algebras have finite-dimensional $\frac12 \mathbb{Z}_{\geq 0}$-graded subspaces and are expected to be quasi-lisse. 
Although $\mathcal{C}_p$ is modeled on $\mathcal{D}^{ch}_{SL(2),-2+\frac{1}{p}}$, which pairs the levels $-2+\frac{1}{p}$ and $-2-\frac{1}{p}$ as in (\ref{CDO}), there is a crucial difference here: $\mathcal{C}_p$ pairs levels $-2-p$ and $-2+\frac{1}{p}$ that actually appear in the quantum version of the geometric Langlands correspondence \cite{Ga}.

To this end, we consider extensions of the vertex algebra $L_{k}( \mathfrak{sl}_2)  \otimes  L_{k'}( \mathfrak{sl}_2)$ such that  
$k+2 = \frac{1}{p}$, $k'+2 = -p$. One of our main results is the following theorem.

\begin{theorem} \label{main1}
For every $p \geq 1$, there is a simple vertex algebra $ \mathcal C_p $  such that  $L_{-2+\frac{1}{p}}( \mathfrak{sl}_2)  \otimes  L_{-2-p}( \mathfrak{sl}_2)$ is conformally embedded into  $ \mathcal C_p $, and we have the following decomposition:
$$ \mathcal C_p =  \bigoplus_{\ell =0} ^{\infty} L_{\widehat{\frak{sl}_2}} (- (2 + p  + p \ell  ) \Lambda_0 +  p  \ell  \Lambda_1) \bigotimes    
  L_{\widehat{\frak{sl}_2}} (- (2 - \frac{1}{ p}  +   \ell  ) \Lambda_0 +   \ell  \Lambda_1). 
 $$
 \end{theorem}
 To prove this result, we use the following vertex algebras constructed in the literature on chiral differential operators \cite{Ara18} (see \cite{AM}):
 $$ \mathcal U_{p} :=  \bigoplus_{\ell =0} ^{\infty} L_{\widehat{\frak{sl}_2}} (- ( 2 + \frac{1}{p} +  \ell  ) \Lambda_0 +    \ell  \Lambda_1) \bigotimes    
  L_{\widehat{\frak{sl}_2}} (- \left( 2- \frac{1}{p}  +  \ell  \right) \Lambda_0 +  \ell  \Lambda_1), 
$$
and 
  $$  \mathcal V_{p}  := H_{DS, f} (\mathcal U_{p}) =  \bigoplus_{\ell =0} ^{\infty}  L^{Vir} (c_{1,-p}, h^{1,-p}_{1, \ell +1}) \bigotimes    
  L_{\widehat{\frak{sl}_2}} (- \left( 2- \frac{1}{p}  +  \ell  \right) \Lambda_0 +  \ell  \Lambda_1).
$$
We notice that
 $$  L^{Vir} (c_{1,-p}, h^{1,-p}_{1, \ell +1}) = L^{Vir}(c_{-p,1}, h^{-p, 1} _{1, np+1}),  $$
 which indicates  that we can apply the inverse Quantum Hamiltonian Reduction (QHR) from  a suitable  category of  $ L^{Vir}_{c_{1,-p}}$--modules to the category $KL_{-p-2} (\mathfrak{sl}_2)$ (cf. Section  \ref{iqhr}).  As a consequence, we get that inverse QHR acts from the category of  $ L^{Vir}_{c_{1,-p}} \otimes L_{-2 + \frac{1}{p}} (\mathfrak{sl}_2)$--modules to the category  of $ L_{-2 - p} (\mathfrak{sl}_2) \otimes L_{-2 + \frac{1}{p}} (\mathfrak{sl}_2)$--modules. Moreover, we prove that  
that $\mathcal C_p$ is exactly  the inverse QHR  of $\mathcal V_{p}$.  

\begin{theorem} \label{general-11} There is an action of the  Lie algebra $\g_0= \mathfrak{sl}_2 \times \mathfrak{sl}_2$ on   $\mathcal V_{p} \otimes \Pi(0) ^{\frac12}$  so that 
$$ \mathcal C_p = \left( \mathcal V_{p} \otimes \Pi(0) ^{\frac12} \right) ^{\mathrm{int}_{\g_0}} \cong  \bigoplus_{\ell =0} ^{\infty} L_{\widehat{\frak{sl}_2}} (- (2 + p  + p \ell  ) \Lambda_0 +  p  \ell  \Lambda_1) \bigotimes    
  L_{\widehat{\frak{sl}_2}} (- (2 - \frac{1}{ p}  +   \ell  ) \Lambda_0 +   \ell  \Lambda_1).
 $$
\end{theorem}

For small values of $p$, we have more  explicit realizations  of $ \mathcal C_p$, which does not use chiral differential operators and results from \cite{Arakawa}.
\begin{itemize}
\item For $p=2$, it was proven in \cite[Section 9]{AKMPP-selecta} that $ \mathcal C_p \cong  M_{(3)} $. This results was obtained in the context of conformal embeddings $\mathfrak{sl}_2 \times \mathfrak{sl}_2 \hookrightarrow \mathfrak{sp}_6$ at $k=-\frac{1}{2}$. Note that  $M_{(3)} ^{\Z_2} \cong L_{-1/2} (\mathfrak{sp}_6)$.
\item For $p=3$,  we prove  that $\mathcal C_3 \cong  L_{-\frac{5}{3}} (G_2)$. 
In order to see this, we first  prove that the rank two  Weyl vertex algebra $M_{(2)}$ is completely reducible as a module for $L_{-5} (\mathfrak{sl}_2) \otimes V^{Vir} _{c_{1,3}}$ (cf. Theorem \ref{m2-dec}). Then we show that
the vertex algebra   $L_{-\frac{5}{3}} (G_2)$ is obtained by the inverse QHR of $M_{(2)}$:
$$ L_{-\frac{5}{3}} (G_2) = \mbox{ Ker} _ { M_{(2)} \otimes \Pi(0) ^{1/2} }S,  $$
which implies that  $\mathcal C_3 \cong  L_{-\frac{5}{3}} (G_2)$ (cf. Theorem  \ref{dec-G2}). 

\item[] As a byproduct, we obtain the decomposition  of conformal embedding $\mathfrak{sl}_2 \times \mathfrak{sl}_2 \hookrightarrow G_2$ at $k=-\frac{5}{3}$ initiated in \cite{AKMPP-16}.


\item[]Since $L_{-\frac{5}{3}} (G_2) $ is an admissible affine vertex algebra, it is quasi-lisse \cite{AK}. So $\mathcal C_3$ is quasi-lisse.

\item  For $p=4$, in Theorem   \ref{c4-inverse} we prove that $\mathcal C_4$ can be realized as the inverse quantum hamiltonian reduction of $L_{-3/2} (\mathfrak{sl}_3)$. The key point for this result is the decomposition of $L_{-3/2} (\mathfrak{sl}_3)$ as an $L_{-6} (\mathfrak{sl}_2) \otimes V^{Vir} _{c_{1,4}}$--module (cf. Theorem   \ref{3/2-dec}). 

\item Next we  identify $\mathcal C_4$ with the affine vertex algebra $\mathcal{W}_{-\frac{23}{4}}(F_4,A_1+\tilde{A}_1)$. Since  it is obtained by the QHR from the admissible affine vertex algebra $L_{-\frac{23}{4}}(F_4)$, we conclude that  $\mathcal C_4$ is quasi-lisse  (cf. Proposition \ref{quasi-lisse}). 

\item For $p=5$, in Theorem \ref{c5}  we established an isomorphism $$\mathcal C_5 \cong \mathcal{W}_{-30+\frac{31}{5}}(E_8,A_4 + A_2).$$
Since $k=-30+\frac{31}{5}$ is admissible level for $E_8$ we also get that $ \mathcal C_5$ is quasi lisse; see Proposition \ref{quasi-lisse-e8}.

\item It is interesting to notice that for $p=2,3,4,5$, the relevant admissible levels of affine vertex algebra needed to describe $\mathcal{C}_p$ are given by $-h^\vee + \frac{h+1}{p}$. For $p \geq 6$, we do not expect   $\mathcal{C}_p$ to be isomorphic to another affine $W$-algebra.

\end{itemize}

By a result of Arakawa and Kawasetsu \cite{AK}, it is known that the character of a quasi-lisse vertex operator algebra satisfies a particular modular linear differential equation (MLDE).
Motivated by their work, we construct explicit MLDEs of minimal order satisfied by ${\rm ch}[\mathcal C_p]$ for $p \in \{ 2,3,4,5\}$ (see Section 8). This also follows from general properties of the associated varieties. 

In the last section we obtained a closed formula for ${\rm ch}[\mathcal{C}_p]$ and discuss modular transformation properties; see Theorem \ref{explicit-char}. 

\subsection{Notation} Throughout we will be using the following notation:

\begin{itemize}

\item $\frak{g}$ will denote a finite-dimensional simple Lie algebra of rank $\ell$, $\frak{h} \subset \frak{g}$ its Cartan subalgebra and we choose $\{\alpha_1,...,\alpha_\ell\}$ a set of simple roots and $\omega_i$, $1 \leq i \leq \ell$ be the corresponding fundamental weights. Subject 
to this, we denote by $Q$ the root lattice, $P$ the weight lattice, and $P_+$ the cone of dominant integral weights. 

\item Given $\lambda \in P_+$, we denote by $V_\frak{g}(\lambda)$ the irreducible finite-dimensional $\frak{g}$-module of highest weight $\lambda$.

\item $Vir={\rm Span}\{ L(n), n \in \mathbb{Z}, c\}$ will denote the Virasoro Lie algebra with the usual bracket relations and central element $c$.

\item $V^{Vir}_c$ denotes the universal Virasoro vertex algebra of central charge $c$.

\item $L^{Vir}_c$ denotes the simple Virasoro vertex algebra of central charge $c$.

\item  Let $L^{Vir}(c,h)$ will denote the irreducible highest weight $V^{Vir}_c$--module with highest weight $h$ (and central charge $c$).

\item We write   $\widetilde L^{Vir}(c,h)$  for the highest weight $V^{Vir}_c$--module with highest weight $h$, which is not necessary irreducible. In other words, a certain quotient of 
the Verma module $M^{Vir}(c,h)$.

\item $\widehat{\frak{g}}=\frak{g} \otimes \mathbb{C}[t,t^{-1}] \oplus \mathbb{C}C$ will denote the affine Lie algebra associated to $\frak{g}$ with central element $C$. 

\item $V^{k}(\frak{g})$ will denote the universal affine vertex algebra associated to $\frak{g}$ of level $k$.

\item $L_{k}(\frak{g})$ will denote the simple affine vertex algebra associated to $\frak{g}$ of level $k$.

\item $L_{\hat{\frak{g}}}(\Lambda)$ will denote the irreducible $\hat{\frak{g}}$-module of highest weight module $\Lambda$. We denote the fundamental weights by $\Lambda_i$, $0 \leq  i \leq \ell$.
We also use sometimes use $L_{\frak{g},k}(\lambda)$, $\lambda \in \frak{h}^*$ to denote the $L_{\hat{\frak{g}}}(\Lambda)$ of level $k$ whose top component is $\frak{g}$-module $L(\lambda)$.

\item We write  $\widetilde L_{\hat{\frak{g}}}(\Lambda)$  for the highest weight $\hat{\frak{g}}$-module of highest weight module $\Lambda$, which is not-necessary irreducible. 

\item $H_{DS}(\cdot)$ will denote the Drinfeld-Sokolov reduction functor \cite{Arakawa}. 

\item $\mathcal{W}^{k}(\frak{g},f)$ will denote the universal affine $W$-algebra of level $k$ and nilpotent element $f$. We also denote by $\mathcal{W}_{k}(\frak{g},f)$ its simple quotient.

\item $M_{(n)}$ denotes the $\frac12$-shifted  $\beta \gamma$-system of rank $n$, also known as the Weyl vertex algebra.

\item Given a conformal vertex algebra $V$ of central charge $c$, and a $V$-module $M$ with finite-dimensional graded subspaces, then
$${\rm ch}[M]={\rm tr}_{M} q^{L(0)-\frac{c}{24}},$$
will denote the character of $M$. 
\item If a semi-simple Lie algebra $\mathfrak{g}$ acts on a vertex algebra $V$ by derivations, let $V^{\mathrm{int}_{\mathfrak{g}}}$ denote the maximal $\mathfrak{g}$-integrable part of $V$, i.e., the direct sum of all finite-dimensional $\mathfrak{g}$-submodules of $V$. The subspace $V^{\mathrm{int}_{\mathfrak{g}}}$ is a vertex subalgebra of $V$.

\item  Vertex algebra $V$ is called quasi-lisse if the  associative variety $X_{L_k(\frak{g})}={\rm Specm}(R_{L_k(\frak{g})})$ admits finitely many symplectic leaves (cf. \cite{AK}).

\end{itemize}

 \section{Inverse Quantum Hamiltonian Reduction  }
 \label{iqhr}

We shall follow realization of $L_k (\mathfrak{sl}_2)$ from   \cite{A-2019}.
Let $k+ 2 = \frac{p'}{p}$. 
For any two co-prime  integers $p, p'$ we set
$$
c_{p', p} := 1 - 6\frac{(p-p')^2}{p p'}.
$$
  For $r,s \in {\Z}$, we set
  $$  h_{r,s} ^{p',p} = \frac{ (sp-rp')^2- (p-p')^2}{4 p p'}. $$
  
  Recall the following well-known fusion rules
    in the category of $L^{Vir}_{c_{1,p}}$--modules for $i \ge 2$ (see \cite{Lin}):
 \begin{equation}  \label{fusion-1-p} 
  {  L^{Vir}(c_{1,p}, h^{1,p} _{1, 2})  \times   L^{Vir}(c_{1,p}, h^{1,p} _{1, i}) = L^{Vir}(c_{1,p}, h^{1,p} _{1, i-1}) + L^{Vir}(c_{1,p}, h^{1,p} _{1, i+1}) }. 
  \end{equation}
Let $\omega$ be  the conformal vector in $V^{Vir} _{ c_{p',p}} $.
Define
\begin{equation}
 \Pi(0) := M_{c,d} (1) \otimes {\C}[{\Z} c]\qquad \text{and} \qquad  \Pi(0)^{\frac{1}{2}} = M_{c,d} (1) \otimes {\C}[{\Z} \frac{c}{2}],
\end{equation}
where $M_{c,d} (1)$ is the Heisenberg vertex algebra generated by  fields $c(z)$ and $d(z)$ with
$$\langle c, c \rangle=\langle d, d \rangle=0, \ \ \langle c, d \rangle=2.$$

We also have  the following $\Z_2$--gradation on $\Pi(0)^{\frac{1}{2}} $
$$ \Pi(0)^{\frac{1}{2}} =   \Pi_0 (0)^{\frac{1}{2} }   + \Pi_1(0)^{\frac{1}{2}},   $$
where $\Pi(0)^{\frac{1}{2}} $ and $\Pi_1(0)^{\frac{1}{2}} = \Pi(0) . e^{\frac{c}{2}}$. 

 Now, result from \cite{A-2019} gives homomorphism of vertex algebras
$$\Phi : V^{k} (\mathfrak{sl}_2 ) \rightarrow V^{Vir} _{ c_{p',p}}   \otimes \Pi(0)  $$
such that
 \begin{equation}
 \begin{split}
e & \mapsto  {\bf 1} \otimes e^{c  }, \label{def-e-3} \\
h & \mapsto  2 \cdot {\bf 1} \otimes  \mu(-1){\bf 1}, \\
f & \mapsto       (k+2) \omega \otimes e^{-c}    -{\bf 1} \otimes \left[\nu(-1)^{2} -  (k+1) \nu(-2) \right] e^{-c },
\end{split}
 \end{equation}
where \bea \mu = \tfrac{1}{2} d  +  \tfrac{k}{4} c , \ \nu = \tfrac{1}{2} d - \tfrac{k}{4} c.  \label{def-mu-nu} \eea
As in \cite{A-2019} we have the screening operator 
 
\bea \label{scr-1} S =  \mbox{Res}_ z Y ( v^{(k)} \otimes e^{\nu} , z), \eea
where $v^{(k)}$  is a singular vector for the Virasoro algebra highest weight module  of  conformal weight $h^{p',p} _{2,1}=\frac{3}{4} k + 1$.

The following result follows from  \cite{A-2019} and \cite{ ACGY}:

 \begin{proposition} \label{opis-1}   Assume that  $k \notin {\Z}_{\ge 0}$, $k \ne -2$  and $n \in {\Z}_{\ge 0}$. Then we have
\bea   L_{\widehat{\frak{sl}_2}} ((k-n ) \Lambda_0 +   n  \Lambda_1)  &\cong&  V^k(\mathfrak{sl}_2).  (v_{1,n+1} \otimes e^{ \frac{n}{2} c}) \nonumber \\
& \cong & \mbox{\rm Ker} _{    L^{Vir} (c_{p',p}, h^{p',p} _{1, n+1} )   \otimes \Pi(0)^{\frac12}} S \nonumber \\   
& \cong &  \left( L^{Vir} (c_{p', p}, h^{1,p}_{1, np +1} )   \otimes \Pi(0)^{\frac12}\right)   ^{\mathrm{int}_{\mathfrak{sl}_2}} .  \eea
In particular,
$L_k(\mathfrak{sl}_2) =   \left( L^{Vir}_{c_{p', p}}   \otimes \Pi(0)^{\frac12}\right)   ^{\mathrm{int}_{\mathfrak{sl}_2}} $.
 \end{proposition}
 
 This result also holds for negative levels such that $k+2 < 0$. Let us look at the negative level case in more details.
 We shall consider special  levels :
 $$k+ 2 = - \frac{1}{p} \quad \mbox{and} \quad  k+2 = - p, $$
 where $p \in  {\Z}_{ >0}$. 
 We have 
 $$ H_{DS} ( L_{\widehat{\frak{sl}_2}} ((-2- \frac{1}{p} -n ) \Lambda_0 +   n  \Lambda_1)   =   L^{Vir} (c_{1, -p}, h^{1, -p} _{1, n+1} ), $$  $$ H_{DS} ( L_{\widehat{\frak{sl}_2}} ((-2- p  -n p  ) \Lambda_0 +   n  p  \Lambda_1)  = L^{Vir} (c_{-p, 1}, h^{-p,1} _{1, n p +1} ). $$
Since  $h_{1, n+1} ^{-1,p} = h_{1, np+1} ^{-p,1}$,  we have $$  L^{Vir} (c_{1, -p}, h^{-1,p} _{1, n+1} ) =   L^{Vir} (c_{-p, 1}, h^{-p,1} _{1, n p +1} ), $$
which implies
 $$ H_{DS} ( L_{\widehat{\frak{sl}_2}} ((-2- \frac{1}{p} -n ) \Lambda_0 +   n  \Lambda_1)  =   H_{DS} ( L_{\widehat{\frak{sl}_2}} ((-2- p  -n p ) \Lambda_0 +   n p   \Lambda_1). $$

\begin{cor}  \label{inverse-negative} Let $k + 2 = -p$ and $n$ a positive integer. Then we have 
 $$  L_{\widehat{\frak{sl}_2}} ((k  -n p ) \Lambda_0 +   n  p  \Lambda_1)   \cong \mbox{\rm Ker} _{    L^{Vir} (c_{-p, 1}, h^{-p,1}_{1, n p +1} )   \otimes \Pi(0)^{\frac12}}   S \cong  \mbox{\rm Ker} _{    L^{Vir} (c_{-p, 1}, h^{1,-p}_{1, n+1} )   \otimes \Pi(0)^{\frac12}}  S , $$
 where  
$S $ is the screening operator given by (\ref{scr-1}). Moreover, we have
$$   L_{\widehat{\frak{sl}_2}} ((k  -n p ) \Lambda_0 +   n  p  \Lambda_1)  =     \left( L^{Vir} (c_{-p, 1}, h^{-p,1}_{1, n p +1} )   \otimes \Pi(0)^{\frac12}\right)   ^{\mathrm{int}_{\mathfrak{sl}_2}} . $$
\end{cor}

Recall that the fusion product of two subsets $A, B$ of a vertex algebra $V$
is
$$A \cdot  B = \mbox{span}_{\C} \{ a_n b \ \vert \ n \in \Z, a, b \in V\}. $$
This product is associative due to Borcherds' identity.


  \section{Construction of the vertex algebra $\mathcal C_{p}$}
  
  
Next result is a slight variation of the result on generic levels mentioned in the introduction as it only requires some minor modification due to rationality of the level involved.
  \begin{theorem}\label{reg-11} 
For each $p \in \Z_{\ge 1}$,  there exists a  simple vertex operator algebra structure on the $L_{-2 + \frac{1}{p} } (\mathfrak{sl}_2) \otimes  L_{-2 - \frac{1}{p} } (\mathfrak{sl}_2) $--module
  $$ \mathcal U_{p} :=  \bigoplus_{\ell =0} ^{\infty} L_{\widehat{\frak{sl}_2}} (- ( 2 + \frac{1}{p} +  \ell  ) \Lambda_0 +    \ell  \Lambda_1) \bigotimes    
  L_{\widehat{\frak{sl}_2}} (- \left( 2- \frac{1}{p}  +  \ell  \right) \Lambda_0 +  \ell  \Lambda_1).
$$
%
$\mathcal U_{p}$ is generated by  the generators of $ L_{-2-1/p} (\mathfrak{sl}_2) \otimes L_{-2+1/p} (\mathfrak{sl}_2)$  and top component of the module    $L_{\widehat{\frak{sl}_2}} (- ( 3 + \frac{1}{p}     ) \Lambda_0 +     \Lambda_1) \bigotimes    L_{\widehat{\frak{sl}_2}} (- \left( 3-  \frac{1}{p}    \right) \Lambda_0 +    \Lambda_1)$ isomorphic to  $V_{\mathfrak{sl}_2} (\omega_1) \otimes V_{\mathfrak{sl}_2} (\omega_1)$.
\end{theorem}
\begin{proof} It is known that $L_{-2 + \frac{1}{p} } (\mathfrak{sl}_2) \otimes  L_{-2 - \frac{1}{p} } (\mathfrak{sl}_2) \hookrightarrow \mathcal{D}^{ch}_{SL_2,-2+\frac1p}$ and that \cite{Ara18}
$$\left(\mathcal{D}^{ch}_{SL_2,-2+\frac1p}\right)^{t sl_2[t] \times t sl_2[t]} \cong \mathcal{O}(SL_2)=\bigoplus_{\ell \in \mathbb{N}_0} V_{\mathfrak{sl}_2} (\ell  \omega_1) \otimes  V_{\mathfrak{sl}_2} (\ell  \omega_1) .$$ Combining this with complete reducibility of the categories $KL^{-2+\frac1p}(\frak{sl}_2)$ and \\ $KL^{-2-\frac1p}(\frak{sl}_2)$ we get 
that $\mathcal{D}^{ch}_{-2+\frac1p}$ is isomorphic to $\mathcal{U}_{p}$ (exactly the same argument is used Section 9.1.6 of \cite{AM} for the generic level).
But it is known that $\mathcal{D}^{ch}_{-2+\frac1p}$ is simple; see again \cite{AM}.

Using ${\rm gr}(\mathcal{D}^{ch}_{G,k}) \cong \mathcal{O}(J_\infty(T^* G))$   (arc algebra of $T^*G \cong \frak{g}^* \times G$), we see that $ \mathcal{D}^{ch}_{G,k}$ is generated by the top component $\mathcal{O}(G)$, and elements of $\frak{g}$, embedded in the degree one component. For $G=SL_2$, it is not hard to see using explicit action of $\frak{sl}_2 \times \frak{sl}_2$ on $\mathcal{O}(SL_2)$ 
(see for instance \cite{FS}) that $V(1) \otimes V(1)$  generates $\mathcal{O}(SL_2)$, and therefore the algebra $\mathcal{D}^{ch}_{SL_2,k}$ is generated by vectors in 
$L_{\widehat{\frak{sl}_2}}( (- 2-k) \Lambda_0) \bigotimes    
  L_{\widehat{\frak{sl}_2}} (\left( -2+k \right)\Lambda_0)$ and  $L_{\widehat{\frak{sl}_2}}( (- 3-k) \Lambda_0 +   \Lambda_1) \bigotimes    
  L_{\widehat{\frak{sl}_2}} (\left( -3+k \right)\Lambda_0+\Lambda_1)$, for all $k$. In particular, everything holds for  $k=-2 + \frac{1}{p}$.
\end{proof}


\begin{remark} When $k \notin \mathbb{Q}$, that is level $k$ is generic, then a version of this result was also obtained more explicitly by I. Frenkel and K. Styrkas  \cite{FS} using Wakimoto-type construction. In their setup, $\mathcal{D}^{ch}_{SL_2,k}$ appears as a subalgebra of their modified regular representation. 
Their proof also relies on the use of semi-simplicity of the category $KL^{k}(\mathfrak{sl_2})$ when $k$ is generic; see also \cite{Mc}. 


\end{remark}

Let $f$ (resp. $\overline f$)  be principal nilpotent element in first  (respond. second) copy of $\mathfrak{sl}_2$.  Arakawa in   \cite{Ara18} (see also commments in  \cite{Mc}) consider the QHR related to $f$ and $\overline f$. Applying his construction on the vertex algebra $\mathcal U_p$, which is simple by the theorem above, we  get:

\begin{cor}  We have the following simple vertex algebras:
  $$  \mathcal V_{p}  = H_{DS, f} (\mathcal U_{p}) =  \bigoplus_{\ell =0} ^{\infty}  L^{Vir} (c_{1,-p}, h^{1,-p}_{1, \ell +1}) \bigotimes    
  L_{\widehat{\frak{sl}_2}} (- \left( 2- \frac{1}{p}  +  \ell  \right) \Lambda_0 +  \ell  \Lambda_1),
$$
 $$ H_{DS, \overline f } (\mathcal U_{p}) =  \bigoplus_{\ell =0} ^{\infty}  L_{\widehat{\frak{sl}_2}} (- ( 2 + \frac{1}{p} +  \ell  ) \Lambda_0 +    \ell  \Lambda_1)   \bigotimes    
  L^{Vir} (c_{1,p}, h^{1,p} _{1, \ell +1}),
$$
 $$ I^{-2 - \frac{1}{p}} _{SL_2}  = H_{DS, f, \overline f } (\mathcal U_{p}) =  \bigoplus_{\ell =0} ^{\infty}   L^{Vir} (c_{1,-p}, h^{1, -p} _{1, \ell +1})   \bigotimes    
  L^{Vir} (c_{1,p}, h^{1,p} _{1, \ell +1}).
$$
 \end{cor} 
The vertex algebra $I^{k} _{SL_2}$ is called the {\em chiral universal centralizer} \cite{Arakawa,FS}, and $H_{DS, \overline f } (\mathcal U_{p})$ is a certain {equivariant} $\mathcal{W}$-algebra as in \cite{Arakawa}. 

\begin{theorem} \label{general-1}   Let $\g_0 = \mathfrak{sl}_2 \times \mathfrak{sl}_2$.  There exist a homomorphism of vertex algebras \bea  V(\g_0) := V^{-p-2} ( \mathfrak{sl}_2 ) \otimes V^{-2 + \frac{1}{p}} ( \mathfrak{sl}_2) \rightarrow   \mathcal V_p  \otimes \Pi(0) ^{\frac12} \label{action-11}\eea  such that
$$ \mathcal C_p =  \left( \mathcal V_{p} \otimes \Pi(0) ^{\frac12}  \right)^{\mathrm{int}_{\g_0} } \cong  \bigoplus_{\ell =0} ^{\infty} W_p(\ell), $$
 where
$$ W_p(\ell) =L _{\widehat{\frak{sl}_2}} (- (2 + p  + p \ell  ) \Lambda_0 +  p  \ell  \Lambda_1) \bigotimes    
  L_{\widehat{\frak{sl}_2}} (- (2 - \frac{1}{ p}  +   \ell  ) \Lambda_0 +   \ell  \Lambda_1), $$
  and   $\g_0$ acts  on  $\mathcal V_{p} \otimes \Pi(0) ^{1/2} $  via the homomorphism   (\ref{action-11}). 
   Moreover, $\mathcal C_p$ is a simple vertex algebra  generated by $W_p(0) + W_p(1)$.
\end{theorem}
\begin{proof}

The decomposition   follows  directly by applying Corollary  \ref{inverse-negative}. 
It remains to prove  the simplicity  of $ \mathcal C_p$.   
Let $$V_p {(\ell)}  =   L^{Vir} (c_{1,-p}, h^{1,-p}_{1, \ell +1}) \bigotimes    
  L_{\widehat{\frak{sl}_2}} (- \left( 2- \frac{1}{p}  +  \ell  \right) \Lambda_0 +  \ell  \Lambda_1) \subset    \mathcal V_{p}. $$
Since $\mathcal U_{p}$ is weakly generated by $ A =  \mathfrak{sl}_2 \times \mathfrak{sl}_2 +  V_{\mathfrak{sl}_2} (  \omega_1)   \otimes V_{\mathfrak{sl}_2} (  \omega_1)   \subset  L_{-2-1/p} (\mathfrak{sl}_2) \otimes L_{-2+1/p} (\mathfrak{sl}_2) \oplus   L_{\widehat{\frak{sl}_2}} (- ( 3 + \frac{1}{p}     ) \Lambda_0 +     \Lambda_1) \bigotimes    L_{\widehat{\frak{sl}_2}} (- \left( 3-  \frac{1}{p}    \right) \Lambda_0 +    \Lambda_1)$, 
one can prove (as in \cite{KW})   that the vertex algebra   $\mathcal V_p$ is generated by  
$$ A^f = {\C} f \times \mathfrak{sl}_2 \oplus {\C} v_{-1} \otimes  V_{\mathfrak{sl}_2} (  \omega_1), $$
where $v_{-1}$ is the lowest weight vector in $V_{\mathfrak{sl}_2} (  \omega_1) $. The space $A^f$ generates $L^{Vir}_{c_{1,-p}} \otimes L_{-2 + \frac{1}{p}}(\mathfrak{sl}_2)$--module 
$V_p(0) \oplus V_p(1)$.

Since the vertex algebra $  \mathcal V_{p}$ is simple, we have that 
   $V_p(1) \cdot V_p(\ell) \ne \{0\}$. 
  We claim that 
  \bea V_p(1) \cdot V_p( \ell)  = V_{p}(\ell +1) \oplus V_p(\ell-1) \quad (\ell \in {\Z}_{\ge 1}). \label{fusion-p} \eea
By  using the  fusion rules for $L_{-2+ \frac{1}{p}}(\mathfrak{sl}_2)$--modules \cite{ACGY}, we  have that \bea V_p(1) \cdot V_p( \ell) \subseteq V_{p}(\ell +1) \oplus V_p(\ell-1) \quad (\ell \in {\Z}_{\ge 1}). \label{fusion-p-2} \eea

Assume that there exists minimal  $\ell_0 \ \in {\Z}_{\ge 1}$ such that the equation (\ref{fusion-p}) does not hold. Then we have two possibilities: 
\begin{itemize}
 \item[(i)] $V_p(1) \cdot V_p( \ell_0) = V_{p}(\ell_0 +1)$; or 
  \item[(ii)] $V_p(1) \cdot V_p( \ell_0) = V_{p}(\ell_0 -1)$.
  \end{itemize}
  If (i) holds, using (\ref{fusion-p-2}), we easily get that $\bigoplus _{\ell = \ell_0} ^{ \infty} V_p(\ell)$ is a  proper ideal in $\mathcal V_p$. But this is not possible, since $\mathcal V_p$ is simple.
  
  If (ii) holds, we get that $\bigoplus _{\ell = 0 } ^{ \ell _0} V_p(\ell)$  is a proper subalgebra of $\mathcal V_p$, which is generated by $V_p(0) + V_p(1)$. This contradicts the fact that $V_p(0) + V_p(1)$ generates $\mathcal V_p$.
  Therefore (\ref{fusion-p}) holds.

  We get that 
  $$ \mathcal C_p =   \bigoplus_{\ell =0} ^{\infty} W_p(\ell), $$
  where   $$ W_p(\ell) =L _{\widehat{\frak{sl}_2}} (- (2 + p  + p \ell  ) \Lambda_0 +  p  \ell  \Lambda_1) \bigotimes    
  L_{\widehat{\frak{sl}_2}} (- (2 - \frac{1}{ p}  +   \ell  ) \Lambda_0 +   \ell  \Lambda_1). $$
Next, we shall prove that the following fusion product holds
   \bea W_p(1) \cdot W_p( \ell)  = W_{p}(\ell +1) \oplus W_p(\ell-1) \quad (\ell \in {\Z}_{\ge 1}). \label{fusion-w-p}\eea

   The fusion rules for    $L^{Vir}_{c_{1,p}}$--modules  (\ref{fusion-1-p}),  and the decomposition of $\mathcal C_p$  implies that
  \bea  W_p(1) \cdot W_p (\ell) \subset W_p(\ell-1) \oplus W_p(\ell +1).  \label{ograda-w-p-2}\eea
   
 Using Proposition  \ref{opis-1}  again  we get that  $$W_p(\ell)  = W_p(0) (    v_{1,1+ \ell p } \otimes v_{\ell} \otimes e^{\frac{ p \ell }{2} c} ) \subset   V_p(\ell) \otimes \Pi(0) ^{1/2},  $$
 where $v_{\ell}$ is a  highest weight  vector in $L_{\widehat{\frak{sl}_2}} (- ( 2+ \ell -\frac{1}{p}     ) \Lambda_0 +   \ell  \Lambda_1) $  and   $v_{1,1+ p\ell} $ is a highest weight vector in    $L^{Vir} (c_{1,-p}, h^{-p,1} _{1, p\ell+1} )$.
 Moreover, the top component of $V_p(\ell)$ is embedded into  $W_p(\ell)$ in the following way
 \bea  V_p(\ell) _{top} \cong   {\C}    v_{1,1+ p \ell} \otimes  V_{\mathfrak{sl}_2} (\ell \omega_1)  \otimes {\C} e^{\frac{\ell p }{2} c} \subset  W_p(\ell). \label{realization-top-fusion} \eea 

The   fusion product 
   $V_p(1) \cdot V_p(\ell) = V_p(\ell -1) + V_p(\ell +1),  (\ell \ge 1)$ of $V_p(0)$--submodules of $\mathcal V_p$,  implies that 
   $$      v_{1,1+ p( \ell \pm 1) }  \otimes v_{ (\ell \pm 1) } \in V_p(1)_{top}  \cdot V_p(\ell)_{top}. $$
Applying (\ref{realization-top-fusion}) we get  $W_p(1) \cdot W_p(\ell)$ contains non-zero vectors 
   $$    v_{1,1 + p (\ell+1)}  \otimes v_{\ell+1} \otimes e^{\frac{(\ell+1) p}{2} c}  \in W_p(\ell +1), $$ 
 $$    v_{1,1 + p (\ell-1)}  \otimes v_{\ell- 1}   \otimes  e^{\frac{p (\ell+1) }{2} c}   = ( (e_2)_{-1} )^p    \left( v_{1,1 + p (\ell-1)}  \otimes v_{\ell- 1}  \otimes   e^{\frac{ p(\ell-1)}{2} c} \right)  \in W_p(\ell-1), $$
 where $e_2 = e^{ c} \in W_p(0)$.
 
 
   This proves that $$  W_p(1) \cdot W_p(\ell) \supset W_p(\ell-1) \oplus W_p(\ell +1),    $$
   and therefore the  fusion product  (\ref{fusion-w-p}) holds.  Since all components $W_p(\ell)$  appearing in the decomposition of $\mathcal C_p$  are simple $W_p(0)$--modules, we get $\mathcal C_p$ is a simple vertex algebra.

   The fusion product  (\ref{fusion-w-p}) also implies that  $\mathcal C_p$ is  generated by $ W_p(0) + W_p( 1)$. 

  \end{proof}
  
  \begin{cor} For $p>2$, $\mathcal{C}_p$ is strongly generated by six generators of conformal weight $1$, generators of $\widehat{sl}_2 \times \widehat{sl}_2$, and $2p+2$ primaries of 
  weight $\frac{p-1}{2}$. Therefore $\mathcal{C}_p$ is a $W$-algebra of type $(1^6,(\frac{p-1}{2})^{2p+2})$. For $p=2$, the algebra is of type $(\frac12)^6$.
    \end{cor}
\begin{proof}  For the generators of weight  $\frac{p-1}{2}$, we choose a basis of the top component of $W(1)^{top}$ of conformal weight $\frac{p-1}{2}$. To see that these are strong generators,
it is enough to observe, using the fusion rules in the theorem, that we have
$$W_p(\ell+1)^{top} \subset W_p(1)^{top}_{-1} (W_p(\ell)^{top}).$$
This inductively gives that all $W_p(\ell)$ are strongly generated by $W_p(1)^{top}$ and $\widehat{sl}_2 \times \widehat{sl}_2$.
For $p=2$, we have $\mathcal{C}_2 = M_{(3)}$ so the type is clearly as stated.
\end{proof}

For even $p$, vertex algebras $\mathcal{C}_p$ are all $\frac12 \mathbb{Z}$-graded, so we can consider the maximal $\mathbb{Z}$-graded subalgebra $\mathcal{C}_p^0$.
\begin{cor} For $p$ even, $\mathcal{C}_p^0$ is simple and we have a decomposition
$$\mathcal{C}_p^0=  \bigoplus_{\ell =0} ^{\infty} W_p(2 \ell). $$
\end{cor}
 
 \begin{remark}
 Another approach for constructing $\mathcal C_p$ is through its identification as the inverse QHR of the equivariant $W$--algebra $ H_{DS,  f } \left(\mathcal{D}^{ch}_{SL_2, k} \right) $ at positive integer level $k=-2+p$.
 It can be shown that, for $p=3$, this equivariant $W$--algebra is isomorphic to the Weyl vertex algebra $M_{(2)}$
 (see Section \ref{m2} below), and for $p=4$,  it is isomorphic to the affine vertex algebra $L_{-\frac{3}{2}}(\mathfrak{sl}_3)$.
  In what follows, we present alternative  realizations of  $\mathcal C_p$ for these two special cases, which may be of independent interest. 
 \end{remark}

\section{The vertex algebra $L_{-\frac{5}{3}} (G_2)$}

In this section we shall see that $\mathcal C_3 \cong L_{-\frac{5}{3}} (G_2)$  and identify  $ L_{-\frac{5}{3}} (G_2)$ as a subalgebra of $M_{(2)}\otimes \Pi(0)^{1/2}$.

 \subsection{Decomposition  of $M_{(2)}$}

\label{m2}

Recall that the Weyl vertex algebra $M_{(m)}$ is generated  by  bosonic fileds $a_i ^{\pm}$ ($i=1, \dots, m$) with $\lambda$--products
$$ [ (a^{\pm} _i )_{\lambda}  (a^{\pm} _j ) ] = 0, \quad [ (a^{+} _i )_{\lambda}  (a^{-} _j ) ] = \delta_{i,j}. $$

The vertex algebra $M_{(m)}$ has the following $\Z_2$--gradation 
$$ M_{(m)} = M_{(m)} ^0 \oplus M_{(m)} ^1, $$
where $M_{(m)} ^0 = L_{-1/2} (\mathfrak{sp}(2m))$,  $M_{(m)} ^1 =  L_{\widehat{\mathfrak{sp}_{2m} }} (-\frac{3}{2} \Lambda_0 + \Lambda_1) $ (cf. \cite{FF}). 

 In this section we prove the complete reducibility of $M_{(2)}$ as $L_{-5} (\frak{sl}_2) \otimes L^{Vir}_{-7}$--module and find an explicit decomposition. 
  Interestingly, this decomposition gives a vertex-algebraic interpretation of the famous $q$--series identity of  Legendre \cite[Example 5.1]{KW-94} on the sum of four triangular numbers.
 
Let $(a)_n=(a;q)_n:=\prod_{i=0}^{n-1} (1-aq^i) $, so that
$$(q)_\infty=(q;q)_\infty =  \prod_{i \geq 1} (1-q^i).$$ 
We shall first identify characters of $L_k(\frak{sl}_2)$--modules for $k =-5$ and  $ L^{Vir}_{c}$-modules with $c=c_{1,3} =-7$.
We are interested in Virasoro lowest weight modules with conformal weights $$ h^{1,3} _{1,i} =  \frac{ (3 i -1)^2 -4}{12} = \frac{ (i-1) (3i +1) }{4}.  $$
\begin{lemma} \label{karakteri} For $ \ell \in {\Z_{\ge 0}}$,  $\ell \ge 2$, we have:
\begin{align*} 
{\rm ch}[L^{Vir} (c_{1,3}, h^{1,3}_{1,i} )] &= q^{7/24}   \frac{1}{(q)_\infty} \left( q^{ h^{1,3} _{1, i} }- q^{h^{1,3} _{1, -i}}\right),
\end{align*}
and
 \begin{align*} {\rm ch}[ L_{\widehat{\frak{sl}_2}} (- 5  -  6 \ell  ) \Lambda_0 +  6 \ell  \Lambda_1)]&=  q^{-5/24}  \frac{1}{(q)_\infty^3} \left(   \sum_{i=0} ^{\ell}  ( 6 i  +1) q^{-i ( 3 i +1) } -   \sum_{i=1} ^{\ell}(6i -1) 
 q^{- i (3 i -1) } \right),
   \\
{\rm ch}[L_{\widehat{\frak{sl}_2}} (- 2  -  6 \ell  ) \Lambda_0 +  ( 6 \ell  -3 ) \Lambda_1)] &=  q^{-5/24}   \frac{1}{(q)_\infty^3}   \sum_{i=1} ^{\ell}  \left(  ( 6 i  -2) q^{-\frac{ (2i -1)    ( 6i - 1)}{4 } } - (6i -4) 
 q^{-\frac{ (2i -1)  ( 6 i - 5)}{4 } } \right).  
\end{align*}
\end{lemma}
\begin{proof}
The first formula is well-known and the proof can be found in \cite{Iohara-Koga}.
Second and third formulas are proven similarly as in \cite[Section 9]{AKMPP-selecta}, or using the resolution of $L_{\widehat{\frak{sl}}_2}(\Lambda)$ via Verma modules; see  \cite[Theorem A]{Mal}. \end{proof}

 Next result shows that $ M_{(2) }$ is completely reducible as an $L_{-5} (\frak{sl}_2) \otimes L^{Vir}_{c_{1,3}}$ --module.
\begin{theorem}   \label{m2-dec} The $\beta \gamma$-system $M_{(2)}$ is completely reducible as an $L_{-5} (\frak{sl}_2) \otimes L^{Vir}_{c_{1,3}}$ --module, and the following decomposition holds
\bea  && M_{(2) }=  \bigoplus_{\ell =0} ^{\infty} V(\ell) , \quad \mbox{where} \label{dec-1} \\
&& V(\ell) =     L_{\widehat{\frak{sl}_2}} (- (5 + 3 \ell  ) \Lambda_0 +  3 \ell  \Lambda_1) \bigotimes    
  L^{Vir} (c_{1,3}, \frac{\ell (3 \ell +4)}{4}  )  \label{m2-1} . \label{def-vl} \eea
  The decomposition satisfies the following fusion product:
\bea  V(1) \cdot V(\ell) = V(\ell -1) + V(\ell +1) \quad (\ell \ge 1).  \label{fusion-g2-3} \eea
\end{theorem}

\begin{proof}

First we notice that there is a homomorphism $\Phi: V^{-5} (\frak{sl}_2) \rightarrow M_{(2)}$  (cf. \cite[Section 3.3]{BMR} uniquely determined by
$$e\mapsto  : a_2 ^+ a_1^-:  +  \frac{1}{\sqrt{3}} : (a_2 ^- )^2: , \ h \mapsto  3 :a_1 ^+ a_1^-: + : a_2^+ a_2^-: , \ f \mapsto  3( : a_1 ^+ a_2^-:  -  \frac{1}{\sqrt{3}} : (a_2 ^+ )^2: ). $$
Let $U = \mbox{Im} (\Phi)$. $U$ is a VOA of central charge $c_1 = \frac{ 3 \cdot \frac{-5}{3}} { -\frac{5}{3} +2 } = -5$.  Then $\mbox{Com} (U, M_{(2)} )$ is a vertex operator algebra of central charge $c=c_{1,3} = -7$.  We conclude that there is conformal embedding
$$ U \otimes    L^{Vir}_ {c_{1,3}}\hookrightarrow M_{(2)}. $$

For $\ell \in {\Z}_{\ge 0}$, we define:
$$w _{\ell} = :(a_1 ^- )^{\ell}:. $$
 We directly see that $w_{\ell}$ is a singular vector and it generates the following highest weight module:
 $$ \widetilde V({\ell}) :=  (L_{-5} (\frak{sl}_2) \otimes L^{Vir}_{c_{1,3}} ). w_{\ell}  =  \widetilde L_{\widehat{\frak{sl}_2}} (- (5 + 3 \ell  ) \Lambda_0 +  3 \ell  \Lambda_1) \bigotimes    
  \widetilde L^{Vir} (c_{1,3}, \frac{\ell (3 \ell +4)}{4}  ), $$
  where      $ \widetilde L_{\widehat{\frak{sl}_2}} (\Lambda)$  is a certain highest weight $L_{-5} (\frak{sl}_2)$--module of highest weight $\Lambda$, and  $\widetilde  L^{Vir} (c_{1,3}, h) $ certain highest weight $L^{Vir}_{c_{1,3}} $--module. (Note that we still do not know whether these modules are irreducible).     
 Thus 
  $$ \sum _{\ell \in {\Z}_{\ge 0}} \widetilde V(\ell) \subset M_{(2)}. $$

In order to prove decomposition (\ref{dec-1}), it suffices to prove  the character identity 
\begin{equation} \label{m2-2-ch} {\rm ch}[M_{(2)}] = \sum_{\ell = 0 } ^{\infty}  {\rm ch}[L_{\widehat{\frak{sl}_2}} (- (5 + 3 \ell  ) \Lambda_0 +  3 \ell  \Lambda_1))]  \cdot {\rm ch}[    
  L^{Vir} (c_{1,3}, \frac{\ell (3 \ell +4)}{4}  )],  
  \end{equation}
  which would also imply that $\widetilde V(\ell) = V(\ell)$ (simplicity of each component). Using the fusion rules  (\ref{fusion-1-p})  for $p=3$, we conclude that
\bea  V(1) \cdot V(\ell) \subset  V(\ell +1) + V(\ell -1).    \label{upper-fr-2}\eea 
  
  Since $V(1) _{top} = \mbox{span}_{\C} \{a_1 ^+, a_1^-, a_2 ^+, a_2 ^-\}$, we  get
  $$ ( a_1 ^+  )_0  w_{\ell} = \ell w_{\ell-1}, \    ( a_1 ^-  )_{-1}   w_{\ell} = w_{\ell+1},   $$
 which implies that 
  $ V(\ell +1) + V(\ell -1) \subset V(1) \cdot V(\ell)$. 
  Using  (\ref{upper-fr-2}) we get   the fusion product (\ref{fusion-g2-3}).
  
  So it remains to prove the character identity (\ref{m2-2-ch}).  This is proved in Lemma  \ref{identitet-prvi} in Section \ref{char-id-section}.
 The proof follows.
\end{proof}

 \subsection{Isomorphism $\mathcal C_3  \cong L_{-5/3} (G_2)$ }
 \label{dec-g2}

Now we shall apply inverse quantum Hamiltonian reduction procedure as in  Section  \ref{iqhr} for  the level $k=-\frac{5}{3}$.

The vertex algebra $M_{(m)}$ has $\Z_2$--twisted module, which we denote by $M_{(m)} ^{tw}$ (cf. \cite{Weiner}) and which is generated by cyclic vector ${\bf 1}^{tw}$. For $m=2$, we get that $M_{(2)}^{tw}$ is a module for $L_{-5} (\mathfrak{sl}_2) \otimes L^{Vir}_{c_{1,3}}$. Moreover, by direct calculation we can check that  $L^{Vir}(c_{1,3}, - \frac{1}{4}) =  L^{Vir} (c_{1,3}, h_{2,1} ^{1,3}) \cong L^{Vir}_{c_{1,3}}. {\bf 1}^{tw}$.

We consider the vertex algebra $M_{(2)} \otimes \Pi(0)^{1/2}$ and its $\Z_2$--twisted module $M_{(2)} ^{tw} \otimes \Pi(0)^{1/2}. e^{\nu}$. The Virasoro highest weight vector $v^{(k)}$ is now represented by ${\bf 1}^{tw} \in M_{(2)} ^{tw}$.
As in (\ref{scr-1}) we have the screening operator

\bea \label{scr-2} S =  \mbox{Res}_ z \mathcal Y ( {\bf 1}^{tw}  \otimes e^{\nu} , z) : M_{(2)} \otimes \Pi (0) ^{1/2}   \rightarrow  M_{(2)} ^{tw} \otimes \Pi(0)^{1/2}. e^{\nu}  \eea
where $\mathcal Y$ is intertwining operator of type
$${ M_{(2)} ^{tw} \otimes \Pi(0)^{1/2}. e^{\nu}  \choose M_{(2)} ^{tw} \otimes  \Pi(0)^{1/2}. e^{\nu}  \ M_{(2)}  \otimes \Pi(0)^{1/2} }, $$
which is transpose of the vertex operator $Y_{W} $ that defines  the structure of a twisted module on  $M_{(2)} ^{tw} \otimes \Pi(0)^{1/2}. e^{\nu}.$ 
\begin{lemma} The space 
$ \mbox{\rm Ker}_{ M_{(2)} \otimes \Pi(0) ^{1/2} } S$ is a vertex subalgebra of $M_{(2)} \otimes \Pi(0) ^{1/2}.$
\end{lemma} \begin{proof} 
This follows from general properties of twisted modules.
Given a $\mathbb{Z}_2$-twisted module $W$ with twisted module map $Y_{W}(u,x)$, we can associate to it so-called twist map $Y_{WV}^W(w,x)$, which can be also viewed as a twisted 
intertwining operator of type ${ W \choose W \ V }$ as above. In the Jacobi identity for $Y_{WV}^W(w,x)$ and $Y_W(u,x)$ where $u \in V^0$ , the Jacobi identity. Now specializing everything to our case with $Y_{W}(u,x)$, where $u$ belongs to the $\mathbb{Z}$-graded part of 
$M_{(2)} \otimes \Pi(0)^{1/2}$, and $w={\bf 1}^{tw}  \otimes e^{\nu}$, after taking the residue in the Jacobi identity twice we get $[Y_W(u,x),S]=0$ for such $u$. This implies that the kernel of 
$S$ is a vertex subalgebra. \end{proof}

 Here is our main theorem in this section:
 \begin{theorem} \label{dec-G2}  Let $\g_0 = \mathfrak{sl}_2 \times \mathfrak{sl}_2$.  Then there exist a homomorphism of vertex algebras \bea  \rho_3:  V(\g_0) := V^{-5} ( \mathfrak{sl}_2 ) \otimes V^{-\frac{5}{3}} ( \mathfrak{sl}_2) \rightarrow   M_{(2)} \otimes \Pi(0) ^{1/2} \label{action-1}\eea  such that
\item[(1)]  $$  W= \mbox{\rm Ker}_{ M_{(2)} \otimes \Pi(0) ^{1/2} } S = \left( M_{(2)} \otimes \Pi(0) ^{1/2} \right)^{\mathrm{int}_{\g_0}} =  \bigoplus_{\ell \in {\Z}_{\ge 0} } W(\ell), $$
 where \bea W(\ell) &=&    L_{\widehat{\frak{sl}_2}} (- (5 + 3 \ell  ) \Lambda_0 +  3 \ell  \Lambda_1)   \otimes  \mbox{Ker}_{    L^{Vir} (c_{1,3}, \frac{\ell (3 \ell +4)}{4}  )  \otimes \Pi(0) ^{1/2} } S  \nonumber \\
 &=&   L_{\widehat{\frak{sl}_2}} (- (5 + 3 \ell  ) \Lambda_0 +  3 \ell  \Lambda_1)   \otimes  L_{\widehat{\frak{sl}_2}} (- (\frac{5}{3}  +   \ell  ) \Lambda_0 +  \ell  \Lambda_1),  \nonumber \eea
and where  $S$ is the  screening operator defined by (\ref{scr-2}), and  $\left( M_{(2)} \otimes \Pi(0) ^{1/2} \right)^{\mathrm{int}_{\g_0} }$ is the maximal integrable part of $ M_{(2)} \otimes \Pi(0) ^{1/2} $ with respect to $\g_0$--action induced  from the homomorphism (\ref{action-1}).
 
  \item[(2)]  The decomposition  in (1)     satisfies the following fusion product:
$$ W(1) \cdot W(\ell) = W(\ell -1) + W(\ell +1) \quad (\ell \ge 1). $$

\item[(3)]  $W$ is a simple vertex  operator algebra.

\item[(4)] $W\cong \mathcal C_3 \cong    L_{-\frac{5}{3}} (G_2) $.
  \end{theorem}
  
  \begin{proof}
  The assertion (1) follows directly from  Propostion  \ref{opis-1} and Theorem   \ref{m2-dec}.  The fusion rules for    $L^{Vir}_{c_{1,p}}$--modules  (\ref{fusion-1-p})  now implies that
  \bea  W(1) \cdot W(\ell) \subset W(\ell-1) \oplus W(\ell +1).  \label{ograda-w-3}\eea
   
 Using Proposition  \ref{opis-1}  again, we get  $$W(\ell)  = W(0) ( v_{3 \ell} \otimes   v_{1,1+ \ell} \otimes e^{\frac{\ell}{2} c} ) = W(0) ( w_{\ell} \otimes e^{\frac{\ell}{2} c})  \subset   V(\ell) \otimes \Pi(0) ^{1/2},  $$
 where $w_{\ell} $ is identified with the vector $v_{3 \ell} \otimes   v_{1,1+ \ell}$, such that $v_{3 \ell}$ is a  highest weight  vector in $L_{\widehat{\frak{sl}_2}} (- (5 + 3 \ell  ) \Lambda_0 +  3 \ell  \Lambda_1) $  and   $v_{1,1+ \ell} $ is a highest weight vector in    $L^{Vir} (c_{1,3}, h^{1,3} _{1, \ell+1} )$.
Applying  the fusion product 
   $V(1) \cdot V(\ell) = V(\ell -1) + V(\ell +1) \quad (\ell \ge 1)$ of $V(0)$--submodules of $M_{(2)}$  and using  (\ref{ograda-w-3})  we get that the fusion product  $W(1) \cdot W(\ell)$ contains non-zero vectors 
   $$ w_{\ell +1} \otimes e^{\frac{\ell+1}{2} c}  \in W(\ell +1),   
 \   w_{\ell-1}  \otimes   e^{\frac{\ell+1}{2} c}  = (e_2)_{-1}  (w_{\ell-1}  \otimes   e^{\frac{\ell-1}{2} c} )  \in W(\ell-1), $$
   where $e_2 = e^{c} \in W(0)$. 
   Indeed, we have $$ (w_{1} \otimes e^{\frac{c}{2}}) _{-1}    (w_{\ell} \otimes e^{\frac{ \ell }{2} c }) =  (w_{\ell+1} \otimes e^{\frac{ \ell  +1}{2} c }) \in  W(1) \cdot W(\ell).   $$ 
 Next we notice that $a_1^+ = \mu_1 f(0)^3 w_{1}$, for certain non-zero constant $\mu_1$, so
 $$a_1 ^+ \otimes e^{\frac{c}{2}} =\mu_1 f(0)^3 w_{1} \otimes e^{\frac{c}{2}} \in W(1),$$
and this implies
   $$ (a_1 ^+ \otimes e^{\frac{c}{2}}) _0     (w_{\ell} \otimes e^{\frac{ \ell }{2} c }) =  \ell (w_{\ell-1} \otimes e^{\frac{ \ell  +1}{2} c }) \in  W(1) \cdot W(\ell).   $$
 We conclude $$  W(1) \cdot W(\ell) \supset W(\ell-1) \oplus W(\ell +1),    $$
   and therefore the assertion (2) holds. Since all components $W(\ell)$ are simple $W(0)$--modules, the assertion (2) implies that $W$ is a simple vertex algebra.
  
  Using (2) and (3) we see that $W$ is a simple vertex algebra generated by $W(0) + W(1)$. More precisely, $W$ is generated by 
  the generators of
  $L_{-5/3} (\mathfrak{sl}_2) \otimes  L_{-5} (\mathfrak{sl}_2)$  which span  $\g_0 = \frak{sl}_2 \times \frak{sl}_2 $ and the    top component of  $ L_{\widehat{\frak{sl}_2}} (- 8  \Lambda_0 +  3  \Lambda_1) \bigotimes    
 L_{\widehat{\frak{sl}_2}} (- \frac{8}{3}    \Lambda_0 +    \Lambda_1)$ which is isomorphic  to $\g_1 =V_{\frak{sl}_2} (3 \omega_1) \otimes V_{\frak{sl}_2} (\omega_1)$.
  This implies that the weight one subspace of $V$, equipped with the bracket $[x,y] = x_{(0)} y$, is the Lie algebra $\g = \g_0 + \g_1 \cong G_2$.
  \end{proof}
 
 \begin{remark}
 Note that in \cite{BMR} the authors also found a homomorphism $\Phi_1:  V^{-5/3} (G_2) \rightarrow M_{(2)} \otimes \Pi(0) ^{1/2}$, but they do not provide a proof of simplicity of  $\mbox{Im} (\Phi_1)$  or decomposition of conformal embedding.
 \end{remark}

  \section{  The vertex algebra $\mathcal C_4$ as inverse QHR of  $L_{-3/2} (\frak{sl}_3)$}
  
 \subsection{The decomposition  of $L_{-3/2} (\mathfrak{sl}_3)$  as $L_{-6} (\mathfrak{sl}_2) \otimes L^{Vir}_{c_{1,4}}$--module }

   We shall dentify characters of $L_k(\mathfrak{sl}_2)$--modules for $k =-6$ and  of $ L^{Vir}_{c}$-modules for $c=c_{1,4} =-\frac{25}{2}$.

Then  $$h^{1,4} _{1,i} = \frac{ (4  i -1)^2 -9}{16} = \frac{ (i-1) (2 i +1) }{2}.$$

Similarly as in Lemma \ref{karakteri}, we have
\begin{lemma} \label{karakteri} Let $ \ell \in {\Z_{\ge 0}}$,  $\ell \ge 2$,  we have:
\begin{align*} 
{\rm ch}[L^{Vir} (c_{1,4}, h^{1,4}_{1,i} )] &= q^{25/48}    \frac{1}{(q)_\infty}   \left( q^{ h^{1,4} _{1,i} }- q^{h^{1,4} _{1,-i}}\right)
\end{align*}
and
 \begin{align*} {\rm ch}[ L_{\widehat{\mathfrak{sl}} _2} (- 6  -  8 \ell  ) \Lambda_0 +  8 \ell  \Lambda_1)] &=  q^{-3/16}   \frac{1}{(q)_\infty^3}  \left(   \sum_{i=0} ^{\ell}  ( 8  i  +1) q^{-i ( 4  i +1) } -   \sum_{i=1} ^{\ell}(8 i -1) 
 q^{- i (4 i -1) } \right),
   \\
 {\rm ch}[ L_{\widehat{\mathfrak{sl}} _2} (- 2  -  8 \ell  ) \Lambda_0 +  ( 8 \ell  - 4  ) \Lambda_1)] &=  q^{-3/ 16}    \frac{1}{(q)_\infty^3}  \sum_{i=1} ^{\ell}  \left(  ( 8  i  - 3) q^{-\frac{ (2i -1)    ( 4 i - 1)}{2  } } - (8 i - 5) 
 q^{-\frac{ (2i -1)  ( 4 i - 3)}{ 2 } } \right). 
\end{align*}

\end{lemma}

 Consider now  the Lie algebra $  \mathfrak{sl}_3$ with standard generators 
 $$e_{\alpha_1}, e_{\alpha_2},  e_{\theta} ,  f_{\alpha_1},  f_{\alpha_2},  f_{\theta} ,  h_{\alpha_1}, h_{\alpha_2} $$
 and the simple affine vertex algebra $L_{-3/2} (\mathfrak{sl}_3)$ of level $-\frac32$.

\begin{theorem}   \label{3/2-dec} $L_{-3/2} (\mathfrak{sl}_3)$ is completely reducible as an $L_{-6} (\mathfrak{sl}_2) \otimes L^{Vir}_{c_{1,4}}$ --module and the following decomposition holds
\bea  && L_{-3/2} (\mathfrak{sl}_3) =  \bigoplus_{\ell =0} ^{\infty} L_{\widehat{\mathfrak{sl}} (2)} (- (6 + 4 \ell  ) \Lambda_0 +  4 \ell  \Lambda_1) \bigotimes    
  L^{Vir} (c_{1,4}, \frac{\ell (2 \ell +3)}{2}  )  \label{m2-1} .\eea
\end{theorem}
\begin{proof}
There is an embedding
 $L_{-6} (\mathfrak{sl}_2) \hookrightarrow L_{-3/2} (\mathfrak{sl}_3)$, given by
 $$ e\mapsto \sqrt{2} ( e_{\alpha_1} + e_{\alpha_2}), \  f\mapsto \sqrt{2}( f_{\alpha_1} + f_{\alpha_2}), \ h =  2  h_{\alpha_1} + 2 h _{\alpha_2}. $$
 Since the commutant $\mbox{Com} (L_{-6} (\mathfrak{sl}_2), L_{-3/2} (\mathfrak{sl}_3))$ has the central charge $$c= -8 - \frac{9}{2} = -\frac{25}{2} =c_{1,4}, $$  we conclude that we have a conformal embedding 
 $L_{-6} (\mathfrak{sl}_2) \otimes L^{Vir}_{c_{1,4}} \hookrightarrow  L_{-3/2} (\mathfrak{sl}_3)$, 
 so we can consider $ L_{-3/2} (\mathfrak{sl}_3)$ as an  $L_{-6} (\mathfrak{sl}_2) \otimes L^{Vir}_{c_{1,4}}$--module.

The rest of the proof now follows from the $q$-series identity given in Lemma \ref{identitet-drugi}, using the same arguments as in the proof of Theorem  \ref{m2-dec}.
\end{proof}

The  inverse reduction functor
 gives that for each $\ell \in {\Z}_{\ge 0}$ we have
  $$ L_{\widehat{\frak{sl}_2}} (- (\frac{7}{4}  +  \ell  ) \Lambda_0 +  \ell  \Lambda_1)  =
   \left (      L^{Vir} (c_{1,4}, \frac{\ell (2 \ell +3)}{2}  )  \otimes \Pi(0) ^{1/2}        \right)^{\frak{sl}_2} . $$
 Using the simplicity of $L_{-3/2}(\mathfrak{sl}_3)$    and  the same arguments as in the proof of Theorem \ref{dec-G2}  we get:
 \begin{theorem} \label{c4-inverse}     Let $\g_0 = \mathfrak{sl}_2 \times \mathfrak{sl}_2$.  There exist a homomorphism of vertex algebras \bea  \rho_4 :  V(\g_0) = V^{-6} ( \mathfrak{sl}_2 ) \otimes V^{-\frac{7}{4}} ( \mathfrak{sl}_2) \rightarrow    L_{-3/2} (\mathfrak{sl}_3)  \otimes \Pi(0) ^{1/2} \label{action-3}, \eea  such that
 \bea    \mathcal C_4&=&  \left( L_{-3/2} (\mathfrak{sl}_3) \otimes \Pi(0) ^{1/2} \right)^{\mathrm{int}_{\g_0}} \nonumber \\
 &=&\bigoplus_{\ell =0} ^{\infty}\left( L_{\widehat{\frak{sl}_2}} (- (6 + 4 \ell  ) \Lambda_0 +  4 \ell  \Lambda_1) \bigotimes    
  L_{\widehat{\frak{sl}_2}} (- (\frac{7}{4}  +  \ell  ) \Lambda_0 +  \ell  \Lambda_1)   \label{m2-12} \right), \nonumber \eea
  where  the Lie algebra  $\g_0$ acts on  $L_{-3/2} (\mathfrak{sl}_3) \otimes \Pi(0) ^{1/2}$ via the homomorphism  (\ref{action-3}).
  Moreover, $\mathcal C_4$ is a simple vertex algebra.
 \end{theorem}
 
 \begin{remark}
 As in Section \ref{dec-g2},  one can see that there exists a twisted $ L_{-3/2} (\mathfrak{sl}_3) $--module  $L^{tw}_{-3/2} (\mathfrak{sl}_3) $ generated by a twisted highest weight vector ${\bf 1}^{tw}$ such that 
 $L^{Vir}_{c_{1,4}} . {\bf 1}^{tw} =  L^{Vir} (c_{1,4}, h^{1,4} _{2,1})$, and that the screening operator (\ref{scr-1}) for $k=-7/4$  can be  extended to an operator 
 \bea \label{scr-3} S =  \mbox{Res}_ z \mathcal Y ( {\bf 1}^{tw}  \otimes e^{\nu} , z) : L_{-3/2} (\mathfrak{sl}_3)  \otimes \Pi (0) ^{1/2}   \rightarrow L^{tw}_{-3/2} (\mathfrak{sl}_3)   \otimes \Pi(0)^{1/2}. e^{\nu}, \nonumber  \eea
such that 
 $\mathcal C_4 =\mbox{Ker} _{  L_{-3/2} (\mathfrak{sl}_3) \otimes \Pi(0) ^{1/2} } S$.  This result will not be used here.
 \end{remark}

  \section{ Realization of $\mathcal C_p$ as an affine $W$--algebra: the cases $p=4,5$}
  
  We identified  $\mathcal C_3$ as the affine vertex algebra $L_{-\frac{5}{3}} (G_2)$. In this section we shall see that   $\mathcal C_p$  for $p=4,5$ admit realizations as simple affine $W$--algebras. Moreover, we shall see that in these cases $\mathcal C_p$ are obtained by QHR from admissible affine vertex algebras, which would imply that both $W$-algebras are quasi-lisse. 
 
\subsection{Affine $W$--algebras}
Let us briefly recall construction of the affine $W$-algebra associated to a nilpotent element $f$ in a simple Lie algebra $\mathfrak{g}$ of level $k$. Given a nilpotent element $f$, we denote by $(e,f,h)$ an $sl_2$-triple with a good grading, so that in particular $\frak{g}=\oplus_{j \in \frac12 \mathbb{Z}} \frak{g}_j$, where $\frak{g}_j=\{ y: [x,y]=j y \}$ with $x:=\frac{h}{2}$. 
We denote by $\mathcal{W}^k(\frak{g},f)$ the universal affine $W$-algebra associated with $\frak{g}$ and such $f$, where $\mathcal{W}^k(\frak{g},f)=H^0_{DS,f}(V^k(\frak{g}))$ and $k$ is the level.
This vertex algebra is conformal, provided that $k \neq -h^\vee$, and its central charge is given by  \cite{KW}
$$c=\frac{k {\rm dim}(\frak{g})}{k+h^\vee}-12 k |x|^2 -\sum_{\alpha \in S_+} (12 m_i^2-12 m_i+2)-\frac12 \frak{g}_{1/2},$$
where $S_+$ is a basis of $\frak{g}_{>0}$ and $m_\alpha$ is the $x$-eigenvalue of $\alpha$, and $(\cdot | \cdot)=\frac{1}{2 h^\vee} \kappa( \cdot | \cdot)$ denotes the normalized Killing form.  
The universal affine $W$-algebra  $\mathcal{W}^k(\g, f)$ is strongly generated by the fields corresponding to homogeneous basis elements of $\frak{g}^f_{\leq 0}=\frak{g}^f \cap \frak{g}_{\leq 0}$ inside the centralizer $\g^{f}$ (cf. \cite{KW}).

   \subsection{Affine $W$-algebra $\mathcal{W}_k(F_4,A_1+\tilde{A}_1)$ } 
    \label{w-f4}
  In this section we shall prove that  $\mathcal C_4$ is quasi-lisse and is realized as the affine $W$--algebra $\mathcal{W}_k(F_4,A_1+\tilde{A}_1)$ with $k=-\frac{23}{4}$.
Here we are using classification of nilpotent orbits and at the same time we slightly abuse notation for $W$-algebras, indicating the nilpotent orbit $A_1+\tilde{A}_1$ 
instead of a nilpotent element  $f_{A_1+\tilde{A}_1}$ that belongs to it. 

  Let $\g$ be the exceptional simple Lie algebra of type $F_4$.  We shall now study vertex algebra   $\mathcal{W}_k(\g, A_1+\tilde{A}_1)$ at admissible level $k=-\frac{23}{4}$.

  Take nilpotent element $f$ in the Lie algebra of type $ A_1+\tilde{A}_1$, and corresponding $\mathfrak{sl}_2$--triple $(e, f, x)$.
  Then $\mbox{ad} (x)$ defines the following gradation of $\frak{g}^{f}$:
  $$ \g^{f} = \g^f _{-3/2} \oplus \g^f _{-1}  \oplus \g^f _{-1/2} \oplus \g^{\natural} $$
  such that $\g^{\natural} = \mathfrak{sl}_2 \times \mathfrak{sl}_2$ (here $\frak{g}^\natural=\frak{g}_0 \cap \frak{g}^f$) and  as  $\g^{\natural}$--modules:
 \begin{align*}
  \g^f_{- 3/2} & = V_{\mathfrak{sl}_2} (\omega_1) \otimes   V_{\mathfrak{sl}_2} (0)  \\   \g^f _{- 1}  & =  V_{\mathfrak{sl}_2} (0) \otimes  V_{\mathfrak{sl}_2} (4\omega_1) + V_{\mathfrak{sl}_2} (0) \otimes  V_{\mathfrak{sl}_2} (0),  \\  \g^f _{- 1/2}  & =  V_{\mathfrak{sl}_2} (\omega_1) \otimes  V_{\mathfrak{sl}_2} (4 \omega_1)   \color{black}.   
  \end{align*}
  Moreover, there is a vertex algebra homomorphism $$ \Phi: V^{8k+40}(\mathfrak{sl}_2)  \otimes V^{k+4} ( \mathfrak{sl}_2)  \hookrightarrow \mathcal{W}^k(F_4,A_1+\tilde{A}_1).$$
  
  Applying \cite{KW}, we get that  $\mathcal{W}^k(F_4,A_1+\tilde{A}_1)$ is strongly  generated by
  \begin{itemize}
  \item   the generators of $\mbox{Im} (\Phi)$,  i.e, by $J^{\{ x \}}$, $x \in \g ^{\natural}$, 
  \item the Virasoro field $L$ of central charge $c =  -\frac{6(3+k)(14+3k)}{9+k}$,
  \item $10$--generators of conformal weight $3/2$, which span $\g^{\natural}$--module  $V_{\mathfrak{sl}_2} (\omega_1) \otimes  V_{\mathfrak{sl}_2} (4 \omega_1)$, these generators can be taken to be primary for $L$ and    $\mbox{Im} (\Phi)$, which span $\g^{\natural}$--module,
  \item $6$ generators of conformal weight $2$ which consists of $L$ and $5$ primary vectors for $L$ and   $\mbox{Im} (\Phi)$, which span 
   $\g^{\natural}$--module $V_{\mathfrak{sl}_2} (0) \otimes  V_{\mathfrak{sl}_2} (4\omega_1) \oplus V_{\mathfrak{sl}_2} (0) \otimes  V_{\mathfrak{sl}_2} (0)$. Note that the multiplicity of the trivial representation is one.
  \item two generators of conformal weight $5/2$, which span $\g^{\natural}$--module
  $V_{\frak{sl}_2} (\omega_1) \otimes V_{\frak{sl}_2} (0)$.
  \end{itemize}
  
  Let $\mathcal{W}_k(F_4,A_1+\tilde{A}_1)$ be the simple quotient of $\mathcal{W}^k(F_4,A_1+\tilde{A}_1)$, and let $f : \mathcal{W}^k(F_4,A_1+\tilde{A}_1)\rightarrow 
  \mathcal{W}_k(F_4,A_1+\tilde{A}_1)$ be the quotient homomorphism.
  Then we have a homomorhism $$\Phi_1 =  f\circ \Phi :   V^{8k+40}(\mathfrak{sl}_2)  \otimes V^{k+4} ( \mathfrak{sl}_2)  \hookrightarrow \mathcal{W}_k(F_4,A_1+\tilde{A}_1).$$
  Central charge of the vertex algebra  $\mbox{Im} (\Phi_1)$ is given by $$c_{sug} = \frac{3 (8k+40)}{ 8k + 42} + \frac{ 3(k+4)}{k+6}.$$   
  By applying \cite[Theorem 3.1]{AMP} we get:
  \begin{itemize}
  \item The embedding  $\mbox{Im} (\Phi_1) \hookrightarrow  \mathcal{W}_k(F_4,A_1+\tilde{A}_1)$ is conformal if and only if $c= c_{sug}$, i.e.,  when
  $k \in \{-\frac{23}{4}, -4, -5,-\frac92 \}$. 
\end{itemize}
  
  \begin{remark}
  This result also  appears in  a much more general setup in \cite{ACLMPP}. Here we have  included details because of reader convenience. 
  \end{remark}
  
  Now we shall consider the case of conformal level $k =-\frac{23}{4}$. Then $\mbox{Im}(\Phi_1)$ is isomorphic to the simple affine vertex algebra   
  $ L_{-6}(\mathfrak{sl}_2) \otimes L_{-\frac{7}{4}} (\mathfrak{sl}_2)$, 
  and we have conformal embeddings
  $$ L_{-6}(\mathfrak{sl}_2) \otimes L_{-\frac{7}{4}} (\mathfrak{sl}_2) \hookrightarrow 
   \mathcal{W}_{-\frac{23}{4}}(F_4,A_1+\tilde{A}_1). $$
   
  \begin{proposition} \begin{itemize}
  \item $\mathcal{W}_{-\frac{23}{4}}(F_4,A_1+\tilde{A}_1)$ is strongly generated by the generators of $\mbox{Im}(\Phi_1)$ and by   $10$ generators of conformal weight $\frac32$. In particilar, $k=-\frac{23}{4}$ is not a collapsing level.
  \end{itemize}
  
  \end{proposition}
  \begin{proof}
   By applying again results from \cite{AMP} we get that  in   $\mathcal{W}_{-\frac{23}{4}}(F_4,A_1+\tilde{A}_1)$ we have:
\begin{itemize}
\item $L= L_{sug}$, where $L_{sug}$ is the Sugawara Virasoro vector of  $L_{-6}(\mathfrak{sl}_2) \otimes L_{-\frac{7}{4}} (\mathfrak{sl}_2)$,
\item If $G$ is a generator of  $\mathcal{W}^{-\frac{23}{4}}(F_4,A_1+\tilde{A}_1) $ of conformal weight $2$ and $\frac52$, $G \notin {\C} L$, when
$G = 0$ in $\mathcal{W}_{-\frac{23}{4}}(F_4,A_1+\tilde{A}_1) $.
\end{itemize}

Since the conformal weight $3/2$ space is irreducible $\g^{\natural}$--module, we have two possibilities
\begin{itemize}
\item[(i)] all generators of conformal weight $3/2$ vanish in   $\mathcal{W}_{-\frac{23}{4}}(F_4,A_1+\tilde{A}_1) $;
\item[(ii)] all generators of conformal weight $3/2$ are non-zero  in  $\mathcal{W}_{-\frac{23}{4}}(F_4,A_1+\tilde{A}_1) $.
\end{itemize}
The first possibility leads to the conclusion that  $\mathcal{W}_{-\frac{23}{4}}(F_4,A_1+\tilde{A}_1)  = L_{-6}(\mathfrak{sl}_2) \otimes L_{-\frac{7}{4}} (\mathfrak{sl}_2)$, but this contradicts Proposition \ref{quasi-lisse} below. Therefore the assertion (ii) holds and the proof follows.

  \end{proof}
  
  Next we have the following result:
  
  \begin{proposition} \label{quasi-lisse}The vertex algebra  $\mathcal{W}_{-\frac{23}{4}}(F_4,A_1+\tilde{A}_1)$ is quasi-lisse.
  \end{proposition}
  \begin{proof} Since $L_{-23/4}(\frak{f}_4)$ is quasi-lisse, using results of Arakawa et al. \cite{AvEM} and  Proposition 6.3 in \cite{AK}, it suffices to check that the condition
  $f_{A_1+\tilde{A}_1}$ is inside the closure of the nilpotent orbit of $\mathbb{O}_q$ (in our case $q=4$). This is equivalent to 
   $f_{A_1+\tilde{A}_1} \in \{ x \in \frak{f}_4: \pi_{\theta_s}(x)^{4}=0 \}$ where $\pi_{\theta_s}$ is the irreducible finite-dimensional representation of $\frak{f}_4$ with highest weight 
   $\theta_s$ where $\theta_s$ is the highest short root. We have $\theta_s=2\alpha_1+3\alpha_2+4 \alpha_3+2 \alpha_4=\omega_1$ and thus $\pi_{\theta_s}=\pi_{\omega_1}$ is the fundamental 
   $26$-dimensional representation. Explicit computation shows that $\pi_{\omega_1}(f_{A_1+\tilde{A}_1})^4=0$ as required.  
  \end{proof}
  
  Next we would like to identify  $\mathcal{W}_{-\frac{23}{4}}(F_4,A_1+\tilde{A}_1)$ with the vertex algebra  $V$ from previous section.  We have that both algebras are simple, and extensions of the vertex algebra  
  $ L_{-6}(\mathfrak{sl}_2) \otimes L_{-\frac{7}{4}} (\mathfrak{sl}_2)$ by weight $\frac32$ generators.
  
  In \cite[Theorem 3.1]{ACKL}, the authors considered the problem of uniqueness of minimal $W$--algebras, which are of the same type as simple $W$--algebra $\mathcal{W}_{-\frac{23}{4}}(F_4,A_1+\tilde{A}_1)$.
  Using the same proof as in \cite{ACKL} we can prove:
  
  \begin{proposition} Assume that   $\mathcal A = \bigoplus _{\ell\in \frac{1}{2} {\Z} } \mathcal  A(\ell)$ is another simple, $\tfrac{1}{2} {\Z}_{\ge 0}$--graded   conformal vertex algebra, such that
\begin{itemize}
  \item[(1)] There is a conformal embeddings  $ L_{-6}(\mathfrak{sl}_2) \otimes L_{-\frac{7}{4}} (\mathfrak{sl}_2) \hookrightarrow  \mathcal A$.
  \item[(2)] $\mathcal A(\frac{1}{2}) = \{0 \}$, and  $\mathcal A$  is strongly generated by $\mathcal A(1) + \mathcal A(3/2)$ and  $\mathcal A(1) \cong \g^{\natural}$, $\mathcal A (\frac{3}{2} )  \cong  V_{A_1} (\omega_1) \otimes V_{A_1} (4 \omega_1)$ as  $\g^{\natural}$--modules.
  \item[(3)] The weight $3/2$ fields in both $\mathcal A$ and  $\mathcal{W}_{-\frac{23}{4}}(F_4,A_1+\tilde{A}_1)$ carry the same OPE with the Virasoro field and weight one fields.
  \end{itemize}
  Then $\mathcal A \cong  \mathcal{W}_{-\frac{23}{4}}(F_4,A_1+\tilde{A}_1)$. 
  \end{proposition}
  
  It is easy to check that our vertex algebra $V$ satisfies the same conditions as  required for the vertex algebra $\mathcal A$ above. Thus we get:
  
  \begin{theorem} \label{ident-w-1} We have 
  $$ V \cong \mathcal{W}_{-\frac{23}{4}}(F_4,A_1+\tilde{A}_1). $$
  \end{theorem}

   \subsection{The vertex algebra $ \mathcal{W}_{-30+31/5}(E_8,A_4 + A_2)$}
   
Now  we shall prove that $\mathcal C_p \cong  \mathcal{W}_{-30+31/5}(E_8,A_4 + A_2)$ for $p=5$. As a consequence, we shall get that $\mathcal C_5$ is quasi-lisse. The approach is similar to those in Subsection \ref{w-f4} above, so we omit some details.
   
Let now $\g$ denoted the exceptional Lie algebra of type $E_8$. 
   Take nilpotent element $f$ in the Lie algebra of type $ A_4 + A_2$, and corresponding $\mathfrak{sl}_2$--triple $(e, f, x)$. 
   Then $\mbox{ad} (x)$ defines the following gradation on $g^f$
  $$ \g^f = \g^f_{-4} \oplus \g^f_{-3} \oplus    \g^f_{-2 }  \oplus \g^f_{-1} \oplus \g^{\natural}$$
 such that $\g^{\natural} = \mathfrak{sl}_2 \times \mathfrak{sl}_2$ and  as  $\g^{\natural}$--modules (this computation is performed with GAP  \cite{gap}):
\begin{align*}
\g^f_{-4} &= V_{\mathfrak{sl}_2} (2\omega_1) \otimes   V_{\mathfrak{sl}_2} (0) ,    \\
 \g^f _{- 3} &=  V_{\mathfrak{sl}_2} (4 \omega_1) \otimes  V_{\mathfrak{sl}_2} (0) + V_{\mathfrak{sl}_2} (3 \omega_1) \otimes  V_{\mathfrak{sl}_2} (\omega_1)  \\
 \g^f _{- 2} &=  V_{\mathfrak{sl}_2} (6 \omega_1) \otimes  V_{\mathfrak{sl}_2} (0) +  V_{\mathfrak{sl}_2} (2 \omega_1) \otimes  V_{\mathfrak{sl}_2} (0) +  V_{\mathfrak{sl}_2} (\omega_1) \otimes  V_{\mathfrak{sl}_2} (\omega_1) \\
  \g^f_{- 1} &=  V_{\mathfrak{sl}_2} (4 \omega_1) \otimes  V_{\mathfrak{sl}_2} (0) +     V_{\mathfrak{sl}_2} ( 5 \omega_1) \otimes  V_{\mathfrak{sl}_2} (\omega_1) +V_{\mathfrak{sl}_2} (0) \otimes  V_{\mathfrak{sl}_2} (0) . 
 \end{align*}
 
 
 One can show that  there is a vertex algebra homomorphism $$ \Phi: V^{15 k+350}(\mathfrak{sl}_2)  \otimes V^{k+22} ( \mathfrak{sl}_2)  \hookrightarrow \mathcal{W}^k(\g,f).$$
  
  Applying \cite{KW} (details will be also presented in  \cite{ACLMPP} )  we get that  $\mathcal{W}^k(\g, f)$ is strongly  generated by
  \begin{itemize}
  \item   the generators of $\mbox{Im} (\Phi)$,  i.e, by $J^{\{ x \}}$, $x \in \g ^{\natural}$, 
  \item the Virasoro field $L$ of central charge $c =  -\frac{6 (12490 + 1095 k + 24 k^2)}{30 + k} $,
  \item $18 $--generators of conformal weight $2$,  
    \item $14$ generators of conformal weight $3$,  
      \item $13$ generators of conformal weight $4$, 
        \item $3$ generators of conformal weight $5$. 
  \end{itemize}
  
  Let $\mathcal{W}_k(\g, f)$ be the simple quotient of $\mathcal{W}^k(\g,f)$, and let $ \pi : \mathcal{W}^k(\g,f)\rightarrow 
  \mathcal{W}_k(\g, f)$ be the quotient homomorphism.
  Then we have a homomorhism $$\Phi_1 =  \pi \circ \Phi :  V^{15 k+350}(\mathfrak{sl}_2)  \otimes V^{k+22} ( \mathfrak{sl}_2) \rightarrow \mathcal{W}_k(\g, f).$$
  Central charge of the vertex algebra  $\mbox{Im} (\Phi_1)$ is given by $$c_{sug} = \frac{3 (15 k+350)}{ 15k + 352} + \frac{ 3(k+22)}{k+24}.$$   
  By applying \cite[Theorem 3.1]{AMP} we get that the embedding  $\mbox{Im} (\Phi_1) \hookrightarrow  \mathcal{W}_k(\g, f )$ is conformal if and only if     $c= c_{sug}$, i.e.,  when
  $k \in \{ -\frac{119}{5}, -\frac{70}{3}, -23 \}$. 
   
   In this section we consider the case  $k= -\frac{119}{5} $.
  Then $\mbox{Im}(\Phi_1)$ is isomorphic to the simple affine vertex algebra   
  $ L_{-7}(\mathfrak{sl}_2) \otimes L_{-\frac{9}{5}} (\mathfrak{sl}_2)$, 
  and we have the conformal embedding
  $$ L_{-7}(\mathfrak{sl}_2) \otimes L_{-\frac{9}{5}} (\mathfrak{sl}_2) \hookrightarrow 
   \mathcal{W}_{k}(\g, f). $$

   if and only if $c= c_{sug}$, i.e.,  when
  $k \in \{ -\frac{119}{5}, -\frac{70}{3}, -23 \}$. 
  
   \begin{remark}
  This result also will appear in   \cite{ACLMPP}  in more general context, where all conformal embeddings of affine vertex algebra to affine $W$--algebras of type $E$ will be classified.
  \end{remark}

  \begin{proposition} Let    $k =-30+ \frac{31}{5} =- \frac{119}{5}$.   Then 
   $\mathcal{W}_{k}(\g,f )$ is strongly generated by the generators of $\mbox{Im}(\Phi_1)$ and by   $12$ generators of conformal weight $2$.  
  \end{proposition}
  \begin{proof}
   By applying again results from \cite{AMP} we get that  in   $\mathcal{W}_{k}(\g, f)$ we have:
\begin{itemize}
\item $L= L_{sug}$, where $L_{sug}$ is the Sugawara Virasoro vector of  $L_{-7}(\mathfrak{sl}_2) \otimes L_{-\frac{9}{5}} (\mathfrak{sl}_2)$.
\item If $G$ is a generator of  $\mathcal{W}^{k}(\g, f) $ of conformal weight $3$, $4$ or $5$  then
$G = 0$ in $\mathcal{W}_{k}(\g,f) $.
\end{itemize}

Since the conformal weight $2$ consist  of ${\C} L + U$, where $U$ is  an irreducible $\g^{\natural}$--module, we have two possibilities
\begin{itemize}
\item[(i)] all generators  from $U$ vanish in   $\mathcal{W}_{k}(\g,f) $;
\item[(ii)] all generators  from $U$    are non-zero  in  $\mathcal{W}_{k}(\g, f) $.
\end{itemize}
The first possibility leads to the conclusion that  $\mathcal{W}_{k}(\g, f)  = L_{-7}(\mathfrak{sl}_2) \otimes L_{-\frac{9}{5}} (\mathfrak{sl}_2)$, but this contradicts Proposition  \ref{quasi-lisse-e8} below. 
Therefore (ii) holds. The proof follows.

  \end{proof}
  
  Next we have the following result:
  
  \begin{proposition} \label{quasi-lisse-e8} The vertex algebra  $\mathcal{W}_{k}(\g, f)$ is quasi-lisse.
  \end{proposition}
  \begin{proof} Since $L_{k}(E_8)$ is quasi-lisse, using results of Arakawa et al. \cite{AvEM} and  Proposition 6.3 in \cite{AK}, it suffices to check that the condition
  $f$ is inside the closure of the nilpotent orbit of $\mathbb{O}_q$ (in our case $q=5$).  This follows using explicit computation as in Proposition \ref{quasi-lisse}.
    \end{proof}
  
Finally, we identify  $\mathcal{W}_{k}(E_8, f)$ with the vertex algebra  $\mathcal C_5$. We know that they are both algebras are simple, and are extensions of the vertex algebra  
  $ L_{-7}(\mathfrak{sl}_2) \otimes L_{-\frac{9}{5}} (\mathfrak{sl}_2)$ and weight $2$ generators. By using similar arguments as in the proof of  Theorem \ref{ident-w-1}  we get:

  \begin{theorem} \label{c5} We have an isomorphism of vertex algebras
  $$ \mathcal C_5  \cong \mathcal{W}_{k}(\g, f). $$
  \end{theorem}

\section{Proofs of character identities}

\label{char-id-section}

\begin{lemma} \label{identitet-prvi} The following  $q$-series identity holds:
\begin{equation} \label{m2-2-ch-appendix}  {\rm ch}[ M_{(2)}] = \sum_{\ell = 0 } ^{\infty}  {\rm ch}[  L_{\widehat{\frak{sl}_2}} (- (5 + 3 \ell  ) \Lambda_0 +  3 \ell  \Lambda_1))]  \cdot {\rm ch}[   
  L^{Vir} (c_{1,3}, \frac{\ell (3 \ell +4)}{4}  )] .  
  \end{equation}
\end{lemma}
\begin{proof} Basic manipulation with $q$-series (here $c=-2$) gives
\begin{align*}
{\rm ch}[M_{(2)}] &=  q^{-c/24 } \prod_{ n= 1} ^{\infty} \frac{1}{(1-q^{n-1/2} )^{4}} = q^{1/12 } \prod_{ n= 1} ^{\infty} (1-q^{n-1/2} )  ^{-4} 
\nonumber \\
  & = q^{1/12 } \frac{1}{(q)_\infty^ 4} \left( \frac{ (q)_\infty^2 }  {(q^{1/2})_\infty} \right) ^4 = q^{1/12}  \frac{1}{(q)_\infty^ 4}  \Delta  (q^{1/2}) ^ 4, 
\end{align*}
where
$$\Delta  (q) = \sum_{ n \in {\Z}_{\ge 0} } q^ { n (n+1) /2} =\frac{  (q^2)_\infty^2 }{ (q)_\infty}. $$

Using  Lemma  \ref{karakteri}  we get that the $q$-character  of the right side of (\ref{m2-1}) is given by
$  \frac{q^{1/12} }{(q)_\infty^4}   \mathcal B[q]  $
where
\bea
\mathcal B[q]   :=  && \sum_{\ell   = 0}^{\infty}   \sum_{i= 0}^{\ell}  (6 i + 1)    \left(  
  q^{\ell (3 \ell  + 2) - i (3 i + 1)}   - 
   q^{ (\ell + 1) (3 \ell + 1 ) - i (3 i + 1)} \right)  \nonumber  \\
   +&&  \sum_{\ell   = 1}^{\infty}   \sum_{i= 1}^{\ell}   (6 i - 1)    \left( - q^{\ell (3 \ell + 2) - i (3 i - 1) }   +  
   q ^{ (\ell  + 1 ) (3 \ell + 1  ) - i (3 i - 1)} \right)  \nonumber   \\
   + &&  \sum_{\ell   = 1}^{\infty}   \sum_{i= 1}^{\ell}   (6 i - 2) \left( 
  q ^{ \frac{ (2 \ell - 1) (6 \ell + 1) - (2 i - 1) (6 i - 1)}{4} }   - 
   q^{\frac{ (2 \ell + 1) (6 \ell - 1) - (2 i - 1) (6 i - 1)}{4} }  \right)  \nonumber   \\
   +&&   \sum_{\ell   = 1}^{\infty}   \sum_{i= 1}^{\ell} (6 i - 4) \left( - q^{ \frac{ (2 \ell - 1) (6 \ell + 1) - (2 i - 1) (6 i - 5)}{4}  }   +
    q^{\frac{ (2 \ell + 1) (6 \ell - 1) - (2 i - 1) (6 i - 5)}{4}} \right)  \nonumber
.  \eea
Therefore  the  proof of the theorem is now reduced to  the following  identity:
\bea  \label{identity} &&  \frac{ (q)_\infty^ {8} } { (q^{1/2})_\infty^4 } = \Delta ( q^{1/2 } ) ^4    = \mathcal B[q]  \nonumber  \eea

By  \cite[Example 5.1]{KW-94} we have that
$$ \Delta(q) ^4 = \sum_{j, k = 0} ^{\infty}  (2 k +1) q^{\frac{ (2 j +1) (2k +1) -1 }{2}  }. $$
  
Next we organize pairs $(j,k) \in \mathbb{Z}_{\ge 0 } ^2$, such that $ k \equiv 0, 2 \mod 3 $, into eight congruent classes. 
Inspection shows that each such pair satisfies one and only one condition as in the table below.
\begin{table}[h]
    \centering
    \begin{adjustbox}{max width=\textwidth}
    \begin{tabular}{c|c|c|c|c}
        \toprule
         & Condition & $\ell \in \mathbb{Z}$ & $i \in \mathbb{Z}$ & Relations \\
        \midrule
        (1) & $\frac{j}{2} - \frac{k}{6} \in \mathbb{Z}_{\geq 0}$ & $\frac{k}{6} + \frac{j}{2}$ & $\frac{j}{2} - \frac{k}{6}$ & $k=3(\ell-i), j = \ell+i$ \\
        & & & & $6i+1 = 3j - k +1$ \\
        (2) & $\frac{j}{2} - \frac{k}{6} \in \mathbb{Z}_{< 0}$ & $\frac{k}{6} + \frac{j}{2}$ & $\frac{k}{6} - \frac{j}{2}$ & $k=3(\ell+i), j = \ell-i$ \\
        & & & & $-(6i-1) = 3j - k +1$ \\
        (3) & $\frac{k}{6} - \frac{j}{2} + \frac{1}{2} \in \mathbb{Z}_{\geq 0}$ & $\frac{k}{6} + \frac{j+1}{2}$ & $\frac{k}{6} - \frac{j-1}{2}$ & $k=3(\ell+i-1), j=\ell-i$ \\
        & & & & $-(6i-4) = 3j - k +1$ \\
        (4) & $\frac{k}{6} - \frac{j}{2} + \frac{1}{2} \in \mathbb{Z}_{< 0}$ & $\frac{k}{6} + \frac{j+1}{2}$ & $\frac{j+1}{2} - \frac{k}{6}$ & $k=3(\ell-i), j=\ell+i-1$ \\
        & & & & $6i-2 = 3j - k +1$ \\
        (5) & $\frac{k}{6} - \frac{j}{2} - \frac{1}{3} \in \mathbb{Z}_{\geq 0}$ & $\frac{k+1}{6} + \frac{j-1}{2}$ & $\frac{k+1}{6} - \frac{j+1}{2}$ & $k=3(\ell+i)+2, j=\ell-i$ \\
        & & & & $-(6i+1) = 3j - k +1$ \\
        (6) & $\frac{k}{6} - \frac{j}{2} - \frac{1}{3} \in \mathbb{Z}_{< 0}$ & $\frac{k+1}{6} + \frac{j-1}{2}$ & $-\frac{k+1}{6} + \frac{j+1}{2}$ & $k=3(\ell-i)+2, j=\ell+i$ \\
        & & & & $6i-1 = 3j - k +1$ \\
        (7) & $\frac{j}{2} - \frac{k}{6} + \frac{5}{6} \in \mathbb{Z}_{\geq 0}$ & $\frac{k+1}{6} + \frac{j}{2}$ & $\frac{j}{2} - \frac{k+1}{6} + 1$ & $k=3(\ell-i)+2, j=\ell+i-1$ \\
        & & & & $6i-4 = 3j - k +1$ \\
        (8) & $\frac{j}{2} - \frac{k}{6} + \frac{5}{6} \in \mathbb{Z}_{< 0}$ & $\frac{k+1}{6} + \frac{j}{2}$ & $\frac{k+1}{6} - \frac{j}{2}$ & $k=3(\ell+i)-1, j=\ell-i$ \\
        & & & & $-(6i-2) = 3j - k +1$ \\
        \bottomrule
    \end{tabular}
    \end{adjustbox}
\end{table}

For $m =1, \dots, 8$, denote by $S_m = \{ (j,k) \in \mathbb{Z}_{\geq 0}^2 \mid (j,k) \text{ satisfies Condition } (m) \}$. Then,

\[
\mathbb{Z}_{\geq 0}^2 \setminus \{ (j, 3s +1) \mid s, j \geq 0 \} = \bigcup_{1 \leq m \leq 8} S_m, \quad S_i \cap S_j = \emptyset \text{ for } i \neq j.
\]
Condition (1) in the table gives the following $q$-series relation:
\begin{align*}
    &\sum_{\ell=0}^{\infty} \sum_{i=0}^{\ell} (6i + 1) q^{\ell(3\ell+2) - i(3i+1)} \quad = \sum_{k=3(\ell-i), j=\ell+i}^{\infty} (3j - k + 1) q^{\frac{(2j+1)(2k+1) -1}{4}} \quad \\
    & = \sum_{(j,k) \in S_1} (3j - k + 1) q^{\frac{(2j+1)(2k+1) -1}{4}}. \\
\end{align*}
Similarly, using Conditions (2)-(8), we obtain:
$$- \sum_{\ell \geq 1} \sum_{i=1}^{\ell} (6i - 1) q^{\ell(3\ell + 2) - i(3i - 1)} \
 = \sum_{(j,k) \in S_2} (3k - j +1) q^{\frac{(2j +1)(2k +1) -1}{4}}, $$
 
$$ -\sum_{\ell \geq 1} \sum_{i=1}^{\ell} (6i - 4) q^{\frac{(2\ell - 1)(6\ell + 1) - (2i - 1)(6i - 5)}{4}} \
    = \sum_{(j,k) \in S_3} (3j - k +1) q^{\frac{(2j +1)(2k +1) -1}{4}}, 
    $$
   
    $$ \sum_{\ell \geq 1} \sum_{i=1}^{\ell} (6i - 2) q^{\frac{(2\ell - 1)(6\ell + 1) - (2i - 1)(6i - 1)}{4}} \
     = \sum_{(j,k) \in S_4} (3j - k +1) q^{\frac{(2j +1)(2k +1) -1}{4}}, $$
    $$-\sum_{\ell \geq0} \sum_{i=0}^{\ell} (6i + 1) q^{(\ell+1)(3\ell + 1) - i(3i + 1)} \
     = \sum_{(j,k) \in S_5} (3j - k +1) q^{\frac{(2j +1)(2k +1) -1}{4}}, $$
    $$ \sum_{\ell \geq 1} \sum_{i=1}^{\ell} (6i - 1) q^{(\ell+1)(3\ell + 1) - i(3i - 1)} \
     = \sum_{(j,k) \in S_6} (3j - k +1) q^{\frac{(2j +1)(2k +1) -1}{4}}, $$
    $$ \sum_{\ell \geq 1} \sum_{i=1}^{\ell} (6i - 4) q^{\frac{(2\ell + 1)(6\ell - 1) - (2i - 1)(6i - 5)}{4}} \
     = \sum_{(j,k) \in S_7} (3j - k +1) q^{\frac{(2j +1)(2k +1) -1}{4}}, $$
    $$ -\sum_{\ell \geq 1} \sum_{i=1}^{\ell} (6i - 2) q^{\frac{(2\ell + 1)(6\ell - 1) - (2i - 1)(6i - 1)}{4}} \
     = \sum_{(j,k) \in S_8} (3j - k +1) q^{\frac{(2j +1)(2k +1) -1}{4}}.$$

 Next we notice that for any symmetric polynomial $\sigma(x,y)$ we have 
\bea \sum_{j,  s = 0} ^{\infty}    j  q^{   \sigma(j,s)  }  =  \sum_{j,  s = 0} ^{\infty}    s  q^{ \sigma(j,s)  }  . \label{sim} \eea
 
This implies that 
\bea \mathcal B[q] & =& \sum_{j, k = 0} ^{\infty}  (3j - k  +1) q^{\frac{ (2 j +1) (2k +1) -1 }{4}  } -   \sum_{j = 0} ^{\infty}   \sum_{s= 0} ^{\infty}   (3j -  (3 s +1)  +1) q^{\frac{ (2 j +1) (2(3 s  +1) ) -1 }{4}  }   \nonumber \\
 &=&   \sum_{j, k = 0} ^{\infty}   (3j - k  +1)   q^{\frac{ (2 j +1) (2k +1) -1 }{4}  }  - \sum_{j,  s = 0} ^{\infty}   3 (j - s)  q^{\frac{  3 (2 s +1) (2 j  +1) -1 }{4}  }  \nonumber \\
 &\underbrace{=}_{{\rm by} \ (\ref{sim})}&   \sum_{j, k = 0} ^{\infty}   (2 k   +1)   q^{\frac{ (2 j +1) (2k +1) -1 }{4}  } =  \Delta ( q^{1/2 } ) ^4. \nonumber \eea
\end{proof}

\begin{lemma} \label{identitet-drugi}  The decomposition (\ref{m2-1}) holds at the level of characters.
\end{lemma}
 
 \begin{proof}
The character of $ L_{-3/2} (\frak{sl}_3)$ is
\bea
{\rm ch}[L_{-3/2} (\frak{sl}_3)]  &= & q^{-c/24 } \frac{ ( q^2)_\infty^8 }{(q)_\infty^8 }=  q^{1/3 } \frac{1}{(q)_\infty^ 4} \left( \frac{ (q^2)_\infty^2 }  {(q)_\infty} \right) ^4   =  q^{1/3 }  \frac{1}{(q)_\infty^ 4}  \Delta  (q) ^ 4, \nonumber 
\eea
where by a Gauss' identity 
$$\Delta  (q) := \sum_{ n \in {\Z}_{\ge 0} } q^ { n (n+1) /2} =\frac{  (q^2 )_\infty^2 }{ (q)_\infty}.$$

Using  Lemma  \ref{karakteri}  we get that the character of the right side of (\ref{3/2-dec}) is given by
$  \frac{q^{1/3} }{(q)_\infty^4}   \mathcal C [q]  $
where
\bea
\mathcal C[q]   :=  && \sum_{\ell   = 0}^{\infty}   \sum_{i= 0}^{\ell}  (8 i + 1)    \left(  
  q^{\ell (4 \ell  + 3) - i (4 i + 1)}   - 
   q^{ (\ell + 1) (4 \ell + 1 ) - i (4 i + 1)} \right)  \nonumber  \\
   +&&  \sum_{\ell   = 1}^{\infty}   \sum_{i= 1}^{\ell}  (8 i - 1)    \left(  
  -q^{\ell (4 \ell  + 3) - i (4 i - 1)}   + 
   q^{ (\ell + 1) (4 \ell + 1 ) - i (4 i - 1)} \right)  \nonumber  \\
  + &&    \sum_{\ell   = 1}^{\infty}  \sum_{i=1} ^{\ell}     ( 8  i  - 3)  \left(  q^{ \frac{ (2 \ell -1) (4 \ell  +1) -  (2i -1)    ( 4 i - 1)}{2  } }  - q^{ \frac{ (2 \ell + 1) (4 \ell  - 1) -  (2i -1)    ( 4 i - 1)}{2  } } \right) \nonumber  \\
  + &&    \sum_{\ell   = 1}^{\infty}  \sum_{i=1} ^{\ell}     ( 8  i  - 5)  \left( - q^{ \frac{ (2 \ell -1) (4 \ell  +1) -  (2i -1)    ( 4 i - 3)}{2  } }   + q^{ \frac{ (2 \ell + 1) (4 \ell  - 1) -  (2i -1)    ( 4 i - 3)}{2  } } \right) . \nonumber 
 \eea
Therefore  the  proof of the theorem is now reduced to  the following $q$-series identity:
\bea  \label{identity} &&  \frac{ (q^2)^{8}_\infty } {(q)_\infty^4 } = \Delta ( q  ) ^4    = \mathcal C[q].  \nonumber  \eea

 
 Using similar calculation as  in previous section, we prove that 
 $$ \mathcal C[q] = \sum_{j,k =0} ^{\infty} ( 4j - 2k +1)  q^{\frac{ (2 j +1) (2k +1) -1 }{2}  }  =   \sum_{j, k = 0} ^{\infty}  (2 k +1) q^{\frac{ (2 j +1) (2k +1) -1 }{2}  }  =   \Delta(q) ^4. $$
 The proof follows.
 \end{proof}

\section{MLDE for $\mathcal C_p$}
\label{MLDE}

We proved that for any  $n \geq 2$,  we have a  simple  vertex algebra structure on the module
  \[
 \mathcal{C}_n=  \bigoplus_{\ell =0} ^{\infty} L_{\widehat{\frak{sl}_2}} ((-n-2 + n \ell  ) \Lambda_0 +  n \ell  \Lambda_1) \bigotimes    
  L_{\widehat{\frak{sl}_2}} (- \left(\frac{2n-1}{n}  +  \ell  \right) \Lambda_0 +  \ell  \Lambda_1),
  \]
  of central charge $c_n=\frac{6(1+n-n^2)}{n}$.
  This vertex algebra is $\frac12 \mathbb{Z}_{\geq 0}$-graded for $n$ even and $\mathbb{Z}_{\geq 0}$-graded for $n$ odd.
  
 To compute the characters of modules $L_{\widehat{\frak{sl}_2}} ( (-n-2 + n \ell  ) \Lambda_0 +  n \ell  \Lambda_1)$, and thus of $\mathcal{C}_n$, we proceed as earlier in 
 the cases of $n=3$ and $n=4$ using \cite[Theorem A]{Mal}. 
 This, combined with a well-known character formula:
 $${\rm ch}[L_{\widehat{\frak{sl}_2}} (- \left(\frac{2n-1}{n}  +  \ell  \right) \Lambda_0 +  \ell  \Lambda_1)]=q^{\frac{(2n-1)}{8}} \frac{(\ell+1)q^{\frac14 n(\ell+2)\ell}}{(q)_\infty^3},$$
 using the decomposition above, gives the following result.
 \begin{theorem} \label{character-full} For $n \geq 2$, we have
 \begin{align*} (q)_\infty^6 q^{c/24} {\rm ch}[ \mathcal{C}_n](q)= & \sum_{i \geq 0, k \geq 0} (1+2(i+k))(1+2in)q^{k(k+1)n+i(2kn+n-1)}  \\
 & -   \sum_{i \geq 0, k \geq 0}  \left(3+2(i+k))(1+2in)q^{k(1+k)n+(i+1)(1+n+2kn)}\right) \\
& +   \sum_{i \geq 0, k \geq 0}   2 (1 + i + k) (1 + n + 2 i n) q^{
 \frac12 + (k^2-\frac12) n + (1 + i) (-1 + n + 2 k n)} \\
 & -  \sum_{i \geq 0, k \geq 0} 2 (1 + i + k) (-1 + n + 2 i n) q^{-\frac12 + (-\frac12+ k^2) n + (1 + i) (1 + n + 2 k n)}.
 \end{align*}
 \end{theorem}
 Observe that we have theta functions of signature $(1,1)$-indefinite appearing in the formula above, 
  making it unclear whether ${\rm ch}[\mathcal{C}_n]$ is modular. For a detailed discussion on the modularity of these series, see the final section.

Recall the the spaces of classical holomorphic modular forms for $\Gamma(1)=SL(2,\mathbb{Z})$ and for the principal congruence subgroup $\Gamma(2)$. There are given by (as graded rings):
\[
\bigoplus_{k=0}^{\infty} M_k(\Gamma(1)) = \mathbb{C}[E_4,E_6],
\]
where $E_{2k}=-\frac{B_{2k}}{(2k)! }+ 
 \frac{2}{(2k - 1)!} \sum_{n \geq 1} \frac{n^{2k-1} q^n}{1-q^n}$ are Eisenstein's series, and by
 \[
\bigoplus_{k=0}^{\infty} M_k(\Gamma(2)) = \mathbb{C}[\theta_2(\tau)^4, \theta_3(\tau)^4],
\]
where $\theta_2(\tau)=\sum_{n \in \mathbb{Z}+\frac12} q^{\frac12 n^2}$ and $\theta_3(\tau) = \sum_{n \in \mathbb{Z}}q^{\frac12 n^2}$ are Jacobi's theta functions.
Let
\[
\Theta_{r,s}(\tau) := \theta_2(\tau)^{4r}\theta_3(\tau)^{4s},
\]
so that $\Theta_{r,s}(\tau)$, with $2r+2s=2k$ span $M_{2k}(\Gamma(2))$.
Then we introduce Serre's $q$-derivatives:
\begin{equation}
\partial_{(k)} f(q) = \left( q \frac{d}{dq} + k E_2(q) \right) f(q),
\end{equation}
and
\begin{equation}
D_q^{(k)} f(q) := \partial_{(2k-2)} \circ \cdots \circ \partial_{(2)} \circ \partial_{(0)} f(q), 
\end{equation}
where we set  $D_q^{(0)} f(q) := f(q)$. 
We are interested in MLDEs of a particular type $\mathcal{D}_q^{(k)} f=0$, where
\begin{equation}
\mathcal{D}_q^{(k)} \equiv D_q^{(k)} + \sum_{r=1}^{k} f_r(q) D_q^{(k-r)}, \quad f_r(q) \in M_{2r}({\Gamma}), 
\end{equation}
where $\Gamma = \Gamma(1)$ or $\Gamma(2)$ depending on the parity of $n$. This type of equations have been studied in the context of vertex algebra characters for a long time (see for instance  \cite{Milas}), and also in connection with Schur's indices of 4d $N=2$ SCFTs  \cite{Beem}.

 Next result would follow if $\mathcal{C}_n$ is quasi-lisse, essentially using the same argument as in \cite{AK}.
 \begin{conj}
 For $n \geq 2$, the character ${\rm ch}[\mathcal{C}_n]$ satisfies an MLDE with holomorphic coefficients for $M_*(\Gamma(1))$ for $n$ odd and with $M_*(\Gamma(2))$ for $n$ even.
 \end{conj}
 In support fo this conjecture we present a few examples where we know that $\mathcal{C}_n$ is quasi-lisse:
 \begin{enumerate}
\item The lowest order MLDE for ${\rm ch}[\mathcal{C}_2]$ is
$$D_q^{(1)} f(q)- \left( \frac{1}{8}\Theta_{1,0}+\frac{1}{8}\Theta_{0,1} \right)f(q)=0.$$
 \item The lowest order MLDE for ${\rm ch}[\mathcal{C}_3]$ is 
  $$D_q^{(2)} f(q)- 75 E_4 f(q)=0.$$
 \item The lowest order MLDE for  ${\rm ch}[\mathcal{C}_4]$ is
 \begin{align*}  
 & D_q^{(3)} f(q)+(-\frac{5}{16} \Theta_{1,0}+\frac{5}{16} \Theta_{0,1}) D_q^{(2)} f(q) +\left(-\frac{77}{2304} \Theta_{2,0}-\frac{89}{1152}\Theta_{1,1}-\frac{17}{2304} \Theta_{0,2} \right) D_q^{(1)} f(q) \\ 
 & +\left(\frac{33}{4096}\Theta_{3,0}+\frac{33}{4096} \Theta_{2,1}-\frac{197}{4096} \Theta_{1,2}+
 \frac{3}{4096} \Theta_{0,3}\right) f(q)=0.
 \end{align*}
  \item The lowest order MLDE for ${\rm ch}[\mathcal{C}_5]$ is
  \begin{align*}
& D_q^{(6)} f(q)-161 E_4 D_q^{(4)} f(q)-\frac{28812}{75} E_6 D_q^{(3)} f(q) -\frac{8965187}{75} E_8 D_q^{(2)} f(q) - \frac{192787364}{125} E_{10}D_q^{(1)} f(q) \\
& + \left( - \frac{5599287}{5} E_4^3 - \frac{48993336}{25} E_6^2 \right)f(q)=0.
 \end{align*}
  \end{enumerate}

 \section{Modularity of the character of $\mathcal{C}_n$}
 Here we discuss modular properties of characters of  $\mathcal{C}_n$. To make formulas more transparent, we multiply 
 the character with $\eta(\tau)^6$ so that only the indefinite-theta function remains, i.e. we are interested in 
 modular properties of
  $$A_n(\tau):= \eta(\tau)^{6} {\rm ch}[\mathcal{C}_n].$$
 
 We first start with the formula obtained earlier.
  \begin{align*} (q)_\infty^6 q^{c/24} {\rm ch}[ \mathcal{C}_n](q)= & \sum_{i \geq 0, k \geq 0} (1+2(i+k))(1+2in)q^{k(k+1)n+i(2kn+n-1)}  \\
 & -   \sum_{i \geq 0, k \geq 0}  \left(3+2(i+k))(1+2in)q^{k(1+k)n+(i+1)(1+n+2kn)}\right) \\
& +   \sum_{i \geq 0, k \geq 0}   2 (1 + i + k) (1 + n + 2 i n) q^{
 \frac12 + (k^2-\frac12) n + (1 + i) (-1 + n + 2 k n)} \\
 & -  \sum_{i \geq 0, k \geq 0} 2 (1 + i + k) (-1 + n + 2 i n) q^{-\frac12 + (-\frac12+ k^2) n + (1 + i) (1 + n + 2 k n)}.
 \end{align*}

 Then, after manipulating the $q$-series and organizing terms based on whether $i$ is even or odd, the RHS of the formula yields
$$(q)_\infty^6 q^{c/24} {\rm ch}[ \mathcal{C}_n](q)= \left( \sum_{i \geq 0, k \geq 0}-  \sum_{i < 0, k < 0} \right)  (1+i+2k))(1+in)q^{k(k+1)n+\frac{i}{2}(2kn+n-1)}.$$ 
Using $c=6(1+\frac{1}{n}-n)$, we multiply the above equation with $q^{-c/24+1/4}=q^{-\frac{1}{4n}+\frac{n}{4}}$, obtaining
$$A_n(\tau)=\left(\sum_{\ell_1,\ell_2 \geq 0}-\sum_{\ell_1,\ell_2 <0}\right)(1+2 \ell_1+\ell_2)(1+n \ell_2)q^{ n\left(\ell_1+\frac{n-1}{2n}\right)^2+n \left(\ell_1+\frac{n-1}{2n}\right)\left(\ell_2+\frac1n\right)} .   $$

 
 \begin{remark} \label{remark.indefinite}
 
 Now letting $Q(x_1,x_2)=nx_1^2+n x_1 x_2$, 
$A=\begin{pmatrix} 2n & n \\ n & 0 \end{pmatrix}$, and $\lambda=(\frac{n-1}{2n},\frac{n+1}{2n}) \in A^{-1} \mathbb{Z}^2/\mathbb{Z}^2$
(this is true only for $n$ odd but; for $n$ even, we first have to take $\tau \to 2 \tau$). With these shifts we can write $A_n(\tau)$ given above as
$$A_n(\tau)=\frac{n}{2} \sum_{(\ell_1,\ell_2) \in \mathbb{Z}^2 + \left(\frac{n-1}{2n},\frac1n \right)} \left( {\rm sgn}(\ell_1)+{\rm sgn}(\ell_2) \right) p(\ell_1,\ell_2) q^{n \ell_1^2+n \ell_1 \ell_2},$$
where $p(\ell_1,\ell_2)= \ell_2(2\ell_1+\ell_2)$. 
It is easy to verify that the polynomial $\ell_2(2\ell_1+\ell_2)$ is $A$-harmonic.
There is yet another expression for $A_n(\tau)$, given by
$$A_n(\tau)=\sum_{\ell_1 \equiv \ell_2 \mod 2 \atop  0 \leq |\ell_2| \leq \ell_1; \ell_1 \geq 0; \ell_2 \in \mathbb{Z}} (1+\ell_1)(1+n \ell_2) q^{\frac{n}{4}(\ell_1+1)^2-\frac{n}{4}(\ell_2+\frac1n)^2},$$
 but this will not be used in the rest of the paper.
\end{remark}

\subsection{Appell-Lerch series} To determine modular properties of $A_n(\tau)$ we need a few facts about generalized Appell-Lerch series $\kappa_\ell(x,y,q)$.
Let (here $x=e(z_1)$ and $y=e(z_2)$)
\begin{equation} \label{sum-AL}
\kappa_\ell(x,y,q)=\sum_{m \in \mathbb{Z}} \frac{q^{\frac{\ell m^2}{2}} x^{\ell m }}{1-x y q^m},
\end{equation}
 with a double series representation (in a certain domain):
$$\kappa_\ell(x,y,q)=\left( \sum_{i,j \geq 0} -\sum_{i,j<0} \right) x^{i+\ell j} y^{i} q^{ \frac{\ell}{2} j^2+i j }.$$ 
Consider now $\kappa_2(x,y,q)$ with shifts and delations along $x$ and $y$ variables 
$$B_n(x,y,q):=\kappa_2(q^{\frac{n}{2}}x,q^{-\frac12} y^n,q^n).$$
Then we have
$$ q^{-\frac{1}{4n}+\frac{n}{4}} \frac{\partial^2}{\partial x \partial y} \left( xy B_{n}(x,y,q) \right)|_{x=y=1} =A_n(\tau).$$

Some special linear combinations of the Appell-Lerch series can
be expressed through theta functions. For $\kappa_2$ this formula is quite nice and reads (see for instance \cite{KW-94,STT})
\begin{equation} \label{appell}
\kappa_2(x,y,q)-\kappa_2(x^{-1},y,q)=- \frac{(q;q)_\infty^3 \theta(x^2;q)}{\theta(xy;q) \theta(xy^{-1};q)}.
\end{equation}
where $\theta(x;q)=\sum_{n \in \mathbb{Z}} (-1)^n q^{\frac{n^2}{2}+\frac{n}{2}} x^n$.
Differentiation of (\ref{sum-AL}), and after specializing at $x=y=1$, yields the following identity
$$ \frac{\partial^2}{\partial x \partial y} \left(xy \kappa_{2}(q^{-n/2} x^{-1},q^{-1/2} y^n,q^n) \right)|_{x=y=1}=- \frac{\partial^2}{\partial x \partial y}\left( xy B_{n}(x,y,q) \right)|_{x=y=1}.$$
Using this relation, together with (\ref{appell}), and some manipulations with $q$-series gives 
\begin{proposition} \label{explicit}
We have for $n \geq 2$, 
$$A_n(\tau)=\frac12 \frac{\partial^2}{\partial x \partial y} \left( \frac{\eta(n \tau)^3 \vartheta_{n,n}(x^2)}{\vartheta_{n,2n-1}(xy^n) \vartheta_{n,2n+1}(xy^{-n})} \right) \bigg|_{x=y=1},$$
where $\vartheta_{n,a}(x;q)=\sum_{n \in \mathbb{Z}} (-1)^k q^{\frac{n}{2}(k+\frac{a}{2n})^2} x^{k+\frac{a}{2n}}$, is a Jacobi form of weight $\frac12$.
\end{proposition}


\begin{theorem} \label{explicit-char}
The series $A_n(\tau)$ is a modular form (on a certain congruence group) of weight $3$. Consequently,  ${\rm ch}[\mathcal{C}_n]$ is modular of weight zero.
\end{theorem}
\begin{proof} Using Proposition \ref{explicit} together with $\vartheta_{n,n}(1)=0$, after differentiation we obtain
that $A_n(\tau)$ can be expressed using specialized (at $x=y=1$) theta series and their first derivatives, so $A_n(\tau)$ is clearly modular. Since we started 
from a Jacobi form of weight one, after two differentiations and specializations we get a modular form of weight $3$. 
Since ${\rm ch}[\mathcal{C}_n]=\frac{1}{\eta(\tau)^6} A_n(\tau)$, the character has weight $0$.
\end{proof}
\begin{remark} (a) In \cite{AKMPP-selecta}, a closed formula
for $A_2(\tau)$ is obtained as a specialization of the Kac-Wakimoto denominator formula for $\widehat{osp(3|2)}$ \cite{KW-94} : 
$$A_2(\tau)=\frac{\eta(\tau)^{12}}{\eta(\tau/2)^6}.$$

(b) For $n=3$, in \cite{AK},  the character of $L_{-\frac{5}{3} } (G_2)$ was computed using the fact that it satisfies the MLDE
$$D_q^{(2)} f(q)- 75 E_4 f(q)=0.$$
This gives
$$A_3(\tau)=\frac{\eta(3 \tau)^6 E^{(3)}_1(\tau)}{\eta(\tau)^2}$$
where $E^{(3)}_1$ is an Eisenstein series of weight $1$.
Both proofs use different techniques compared to Theorem \ref{explicit-char}.

(c) Analysis of character-like solutions of MLDEs in Section 8, for low $p$, suggests that $\mathcal{C}_p$ admits a unique ordinary module.

\end{remark}

Based on $p=2,3,4$ and $5$, cases studied previously and Theorem \ref{explicit-char} we expect
 \begin{conj} Vertex operator algebra $\mathcal{C}_p$ is quasi-lisse for all $p$.
 \end{conj}

\section{Future work}

There are several directions that we we would like to pursue.

\begin{enumerate}

\item
    We plan to generalize some results  of this paper by extending them to higher-rank simple Lie algebras at levels $-h^\vee+\frac{1}{p}$ and $-h^\vee-p$, starting with $\frak{sl}_3$. Although some results on the structure of such representations at negative levels are missing or they are subject to values of Kazhdan-Lusztig polynomials, we believe that higher rank simple vertex algebras $\mathcal{C}_{p,\frak{g}}$ generalizing $\frak{g}=\frak{sl}_2$, can be rigorously constructed using methods of this paper for every simple $\frak{g}$. If $\frak{g}$ is of ADE type, we expect that we have a simple
    vertex algebra structure on
 $$\mathcal{C}_{p,\frak{g}}= \bigoplus_{\lambda \in P^+} L_{-p-h^\vee, p \lambda}(\frak{g}) \otimes L_{-h^\vee+\frac{1}{p},\lambda^*}(\frak{g}).$$
 Again, unlike $\mathcal{D}^{ch}_{G,-h^\vee+\frac{1}{p}}$, this space is $\frac12 \mathbb{Z}_{\geq 0}$-graded with finite-dimensional graded subspaces.       
 For $p=2$ and $\frak{g}=\frak{sl}_3$, we expect $\mathcal{C}_{2,\frak{sl}_3} \cong L_{-\frac52}(F_4)$, from the Deligne series \cite{AK}.

\item  As hinted in the introduction, studying non-ordinary modules for $\mathcal{C}_p$ would be interesting, as it provides a direct connection to minimal representations. Another direction for exploration is the construction of logarithmic modules, as in \cite{ACGY}.


   \item  The modularity of characters discussed in the final section suggests new families of indefinite-type modular theta functions with harmonic coefficients arising from the ``numerator'' of ${\rm ch}[\mathcal{C}_{p,\mathfrak{g}}]$ with harmonic coefficients (see Remark \ref{remark.indefinite}). Related higher rank $q$-series with polynomial coefficients have been previously studied by S. Zwegers and others \cite{RZ,Zwegers}. 

\end{enumerate}

\vskip 5mm

 \paragraph{\textbf{Acknowledgements}}  

 D.A. is partially supported by the Croatian Science Foundation under the project IP-2022-10-9006.
We would like to thank Igor Frenkel for bringing \cite{FS} to our attention a while ago, and Sander Zwegers for discussion regarding indefinite-theta functions.

\vskip3pt


\end{document}